\documentclass[a4paper,11pt]{scrartcl} 

\usepackage{amsmath} 
\usepackage{amssymb}
\usepackage{amsthm} 
\usepackage[ngerman,english]{babel}
\usepackage{cite}
\usepackage{color}
\usepackage{enumitem}
\usepackage{graphicx} 
\usepackage[shortcuts]{extdash}
\usepackage[utf8]{inputenc} 

\usepackage{typearea}
\usepackage{microtype}

\usepackage[                
pdftex,                  % Ausgabe-Medium: PDF
hidelinks,
colorlinks=false,         % farbige Links in der Bildschirm-Version?
linkcolor=black,         % Farbe für Querverweise
citecolor=black,         % Farbe für Zitierungen
urlcolor=black,          % Farbe für Links
bookmarks=true
]{hyperref}

\newtheoremstyle{break}%
{}{}%
{\itshape}{}%
{\bfseries}{}%  % Note that final punctuation is omitted.
{\newline}{}

\theoremstyle{plain}
\newtheorem{thm}{Theorem}[section]
\newtheorem{lem}[thm]{Lemma}

\newtheorem{prop}[thm]{Proposition}

\theoremstyle{definition}

\newtheorem{ex}[thm]{Example}
\newtheorem{ass}[thm]{Assumption}

\theoremstyle{remark}

\theoremstyle{break}
\newtheorem{alg}[thm]{Algorithm}

\newcommand{\prox}{\operatorname{prox}}
\newcommand{\R}{\mathbb{R}}
\newcommand{\pred}{\operatorname{pred}}
\newcommand{\ared}{\operatorname{ared}}
\newcommand{\dist}{\operatorname{dist}}

\allowdisplaybreaks

\providecommand{\keywords}[1]
{
	\small	
	\noindent
	\textbf{\textit{Keywords---}} #1
}

\providecommand{\ams}[1]
{
	\small	
	\noindent
	\textbf{\textit{AMS Subject Classifications---}} #1
}

\begin{document}
\title{Efficient Regularized Proximal Quasi\-/Newton Methods for Large-Scale Nonconvex Composite Optimization 
Problems}
\author{
    Christian Kanzow%
    \thanks{University of W\"urzburg, Institute of Mathematics,
        Emil-Fischer-Str.\ 30, 97074 W\"urzburg, Germany; kanzow@mathematik.uni-wuerzburg.de,
        theresa.lechner2@mathematik.uni-wuerzburg.de}
    \and
    Theresa Lechner\footnotemark[1]
}

\date{January 10, 2022}
\maketitle

\begin{abstract}
	\small 
	\noindent Optimization problems with composite functions consist of an objective
	function which is the sum of a smooth and a (convex) nonsmooth term.
	This particular structure is exploited by the class of proximal
	gradient methods and some of their generalizations like 
	proximal Newton and quasi-Newton methods. 
	In this paper, we propose a regularized proximal quasi-Newton method whose
        main features are: (a) the method is globally convergent to stationary points,
        (b) the globalization is controlled by a regularization parameter, no line search is required,
        (c) the method can be implemented very efficiently based on a simple observation
        which combines recent ideas for the computation of quasi-Newton
        proximity operators and compact representations of limited-memory quasi-Newton updates.
        Numerical examples for the solution of convex and nonconvex composite optimization
        problems indicate that the method outperforms several existing methods.
\end{abstract}


\ams{49M15, 49M37, 65K05, 65K10, 90C06, 90C26, 90C30, 90C53}
\normalsize

\section{Introduction}

We consider the problem 
\begin{equation} \label{eq:problem}
	\min_x \psi(x):= f(x)+\varphi(x), 
\end{equation}
where $f:\mathbb{R}^n\to\mathbb{R}$ is continuously differentiable and $\varphi:\mathbb{R}^n\to\mathbb{R}$ is convex. In this formulation, the objective function $\psi$ is neither smooth nor convex, so a wide class of problems is covered, including problems in machine learning, compressed sensing, signal processing, and statistics. Although the assumption that $\varphi$ is real-valued seems quite restrictive, the above formulation still comprises a considerably high number of applications in the above fields.

Probably the most prominent example in composite optimization, especially in the context of signal 
processing and compressed sensing, is the $\ell_1$-regularized least squares problem \cite{beck2009fast,figueiredo2007gradient,kim2007interior,wen2010fast}, also called basic pursuit denoising,
which tries to solve the problem
$$
 \min_x \frac{1}{2} \|Ax-b\|_2^2+\lambda\|x\|_1,
$$
where the quadratic term with $A\in\R^{n\times m}$, $b\in\R^m$ has the purpose to find an approximate solution of $Ax\approx b$, whereas the $\ell_1$-term with a regularization parameter $\lambda>0$ controls the sparsity of the solution. More details on this problem can be found in \cite{figueiredo2007gradient}. A wide class of more general applications combines this regularization $\varphi(x)=\lambda \|x\|_1$ with arbitrary convex \cite{beck2009fast,combettes2005signal,hale2008fixed,byrd2016inexact} or nonconvex \cite{milzarek2014} functions $f$
which are also covered by our setting. In particular, this includes problems with different loss functions like the logistic loss
$$
   f(x):=\tfrac 1m \sum_{i=1}^m \log\big(1+\exp(a_i^Tx)\big),
$$ 
see \cite{koh2007interior, lee2014proximal,byrd2016family}, or the nonconvex Student's $t$-loss 
$$
   f(x):=\tfrac 1m \sum_{i=1}^m \log\big( 1+(a_i^Tx-b_i)^2\big),
$$ 
for some data $a_i\in\R^n,b_i\in\R$, cf.\ \cite{aravkin2012robust2,milzarek2014}. These loss problems are typically used to classify data or reconstruct incomplete or blurred data under sparsity constraints.
For more applications of problem \eqref{eq:problem}, we refer to \cite{kanzow2021globalized,combettes2005signal,bonettini2017convergence} and references therein.
\bigskip

\noindent
There are countless algorithms for determining solutions of composite optimization problems. These include semismooth Newton methods \cite{milzarek2014,muoi2013semismooth,li2018highly}, interior point methods \cite{kim2007interior,koh2007interior}, trust-region methods \cite{aravkin2021proximal,chen2020trust}, fixed point methods \cite{chen2013primal,chen2016fixed}, or reformulations into a smooth problem with a forward backward envelope \cite{stella2017forward,themelis2018forward}, to name just a few. The focus in this paper, however, is on proximal-type methods, as these offer a very efficient way for solving many composite optimization problems.

Proximal-type methods for the solution of composite optimization problems
trace back to the generalized proximal-point method by Fukushima and Mine \cite{fukushima1981generalized}. The general purpose algorithm for solving \eqref{eq:problem} is to use a quadratic approximation of the smooth part $f$ and to solve, in each step, a problem of the form
\begin{equation}\label{eq:general_subproblem}
	\min_x f(x^k)+\nabla f(x^k)^T(x-x^k)+\frac 12(x-x^k)^T H_k(x-x^k)+\varphi(x),
\end{equation}
where $x^k$ denotes the current iterate. A crucial point for developing such algorithms is the choice of the matrix $H_k\in\R^{n\times n}$.

First-order methods use $H_k$ as a positive multiple of the identity matrix and are often referred to as proximal gradient methods. In many cases, $H_k$ is constant over the complete algorithm and does not depend on the iteration. The main advantage of these algorithms is that the solution of the subproblems \eqref{eq:general_subproblem} can be done very efficiently or sometimes even analytically
(depending on the function $ \varphi $). A prominent method of this kind is the Iterative Shrinkage Threshold Algorithm \cite{beck2009fast} and its separable extension \cite{tseng2009coordinate}. Many improvements are possible to accelerate this approach \cite{beck2009fast,gu2018inexact,nesterov2013gradient,wright2009sparse}.

Proximal quasi-Newton and variable metric proximal methods choose $ H_k $ by using a suitable updating technique, hence
$ H_k $ changes from iteration to iteration, and the quadratic part in the subproblem
\eqref{eq:general_subproblem} typically yields a much better approximation of the nonlinear function $ f $
than for the simple choice in proximal gradient methods.
On the other hand, this more advanced choice of $ H_k $ makes the subproblem \eqref{eq:general_subproblem}
more difficult to solve, in particular, analytic solutions are usually no longer available. In order to 
deal with this disadvantage, suitable methods therefore allow to solve these subproblems only inexactly.
Global convergence results for these proximal quasi-Newton methods are available in  \cite{bonettini2016variable,bonettini2017convergence,fountoulakis2018flexible,stella2017forward,jiang2012inexact},
which are based on different inexactness criteria, line search techniques, and appropriate assumptions regarding the choice of the sequence $ \{ H_k \} $ (usually uniform boundedeness and positive definiteness).

Using (at least approximate) second-order information in $H_k$ yields the class of proximal Newton methods \cite{becker2012quasi,becker2019quasi,lee2018distributed,ochs2019adaptive,lee2019inexact}. The standard technique to ensure global convergence is to combine the solution of the subproblems with some backtracking strategy.
Similar to proximal quasi-Newton methods, these proximal Newton approaches often use different criteria to solve \eqref{eq:general_subproblem} only inexactly. Despite having suitable global convergence properties, they
also inherit the local fast convergence known from Newton-type methods under certain assumptions, see \cite{mordukhovich2020globally,byrd2016family,ghanbari2018proximal,scheinberg2016practical,lee2014proximal,yue2019family} for several realizations.

In this article, we present a different approach, in which $H_k$ is the sum of a matrix $B_k$ and a multiple $\mu_kI$ of the identity matrix for some regularization parameter $\mu_k>0$. The purpose is to chose $B_k$ as a
(limited memory) quasi-Newton approximation to the Hessian $\nabla^2 f(x^k)$ in the current iterate and to increase or decrease $\mu_k$ according to a trust-region-type framework, depending on the merit of the iteration. As a consequence, the method gets along without using a classical line search approach, which turns out to be more efficient in numerical examples. Moreover, and this is a central point of our contribution, if
$B_k$ is chosen as a limited memory quasi-Newton approximation of $\nabla^2 f(x^k)$, we combine the theory of Becker et al. \cite{becker2019quasi} with the compact representation of these limited memory quasi-Newton methods
in order to get a very efficient solution technique for the resulting subproblems \eqref{eq:general_subproblem}. To the authors' knowledge, there exist only few publications dealing with limited memory matrices and the advantages of their compact representation for proximal-type methods, e.g.\ \cite{lee2018distributed,karimi2017imro}. The combination with the results in  \cite{becker2019quasi} outline the benefits and makes this technique applicable to a wider class of applications, especially for large scale problems. 

The idea of combining the regularization and (proximal) quasi-Newton techniques goes back to the corresponding methods for smooth problems ($\varphi=0$), where the subproblem \eqref{eq:general_subproblem} reduces to $H_k(x-x^k)=-\nabla f(x^k)$, at least if $H_k$ is positive semidefinite. Some improvements \cite{steck2019regularization,ueda2010convergence,ueda2014regularized,li2004regularized} have been made similar to our approach. Trust-region methods for nonsmooth problems in the form of \eqref{eq:problem} are also considered in different papers \cite{chen2020trust,fletcher1982model,kim2010scalable,qi1994trust}.
Techniques for the regularization of proximal quasi-Newton methods are investigated in several variations in literature. The proximal Newton method by Lee, Sun, Saunders \cite{lee2014proximal} does not explicitly use a regularization parameter, but the application to proximal quasi-Newton methods covers this idea if the regularization parameter tends to zero. A similar approach is used in the authors' work in \cite{kanzow2021globalized}.
Regularization of $B_k$ by adding a positive multiple of the identity matrix is also used in \cite{ghanbari2018proximal,scheinberg2016practical}, but convergence is only shown for convex functions $f$. Approaches for solving the subproblems inexactly are investigated in \cite{lee2019inexact,yue2019family}. 
Finally, we mention that the essence of the proximal quasi-Newton method from Karimi and Vavasis \cite{karimi2017imro} is similar to our approach. However, they only consider $\ell_1$-regularized least squares problems and allow $H_k$ to be a 'diagonal minus rank-1'-matrix. Furthermore, they do not use a regularization of $H_k$. Their theoretical approach is generalized by the work of Becker et al. \cite{becker2019quasi}.
We outline the main differences of these methods to the current one after stating our algorithm in Section \ref{sec:method}.

The paper is organized as follows. We first give an overview of some background material in Section~\ref{sec:prelims}. The regularized proximal quasi-Newton method itself is presented in 
Section~\ref{sec:method}. Global convergence of this method is shown in Section~\ref{sec:globalconvergence}
under fairly mild assumptions in the trust-region framework. In addition, under an error bound assumption we prove that a sequence generated by our method is convergent and summable.
Section~\ref{sec:limited} describes the new trick for an efficient
solution of the resulting subproblems \eqref{eq:general_subproblem} if $ B_k $ is computed by a limited memory quasi-Newton technique.
Numerical results and comparisons with some standard solvers are provided in Section~\ref{sec:numeric} with a focus on proximal-type methods.
We conclude with some final remarks in Section~\ref{sec:conclusion}.

Notation:
The set of all symmetric positive definite matrices in $\R^{n\times n}$ is denoted by $\mathbb{S}_{++}^n$. We write $A\succeq B$ or $A\succ B$, if the matrix $A-B$ is positive semidefinite or positive definite, resp.
For a symmetric matrix $H\in\R^{n\times n}$, $\lambda_{\min}(H)$ and $\lambda_{\max}(H)$ denote the smallest and largest eigenvalue of $H$, respectively. Furthermore, $\|\cdot\|$ and $\langle\cdot,\cdot\rangle$ are the Euclidean norm and scalar product, while $\|\cdot\|_H$ and $\langle\cdot,\cdot\rangle_H$ denote the norm and scalar product with respect to $H\in\mathbb{S}_{++}^n$, i.e.\ $\langle x,y\rangle_H=x^THy$ and $\|x\|_H=\sqrt{\langle x,x\rangle_H}$.
We write $x_{\mathcal{I}}$ to describe the subvector of $x\in\R^n$ consisting of all entries $x_i$ with $i\in \mathcal{I}$.

\section{Preliminaries}\label{sec:prelims}

This section summarizes some background material and states a preliminary result which will 
be used in order to derive and investigate our regularized proximal quasi-Newton method.

The \emph{subdifferential} $\partial \varphi(x)$ of a convex function $\varphi:\R^n\to\R$ in a point $x\in\R^n$ is defined as
$$
	\partial\varphi(x):=\big\{ s\in\R^n \mid \varphi(y)\geq \varphi(x)+s^T(y-x) \ 
	\forall y\in\R^n\big\}.
$$
Some properties of this subdifferential are summarized in the following proposition, cf.\ the 
classical monograph \cite{rockafellar2015convex} by Rockafellar for more details.

\begin{prop}\label{prop:subdifferential}
	Let $\varphi:\R^n\to\R$ be convex. Then the following statements hold:
	\begin{enumerate}[label=(\alph*)]
		\item $\partial \varphi(x)\neq \emptyset$ for every $x\in\R^n$ \cite[Theorem 23.4]{rockafellar2015convex}.
		\item $\partial\varphi$ maps bounded sets onto bounded sets \cite[Theorem 24.7]{rockafellar2015convex}.
		\item Let $\{x^k\},\{s^k\}\subset\R^n$ be sequences such that $x^k\to x^*$, $s^k\to s^*$ and $s^k\in\partial\varphi(x^k)$ for all $k\in\mathbb{N}$. Then $s^*\in\partial\varphi(x^*)$ (closedness of the subdifferential) \cite[Theorem 24.4]{rockafellar2015convex}.
		\item $x^*\in\arg\min \varphi$ if and only if $0\in\partial\varphi(x^*)$ (Fermat's rule) \cite[Theorem 16.3]{bauschke2017convex}.
	\end{enumerate}
\end{prop}

\noindent
Note that, in general, parts (a) and (b) do not hold if $\varphi$ is extended-valued.

The basis of proximal-type methods is the proximity operator, introduced by Moreau \cite{moreau1965proximite}. For a convex function $\varphi:\R^n\to \R$ and a positive definite matrix $H\in\mathbb{S}_{++}^n$, the \emph{proximity operator} with respect to $H$ is the mapping
$$
	x\mapsto \prox_{\varphi}^H(x):=\underset{y}{\arg\min} \Big\{ \varphi(y) + \frac 12 (y-x)^TH(y-x)\Big\}.
$$
Since $H$ is positive definite, the regularization $\varphi(y)+\tfrac 12 (y-x)^TH(y-x)$ is strongly convex. Hence, it has a unique minimizer for every $x\in\R^n$, thus the proximity operator is 
well-defined. If $H$ is the identity matrix, we simply write 
\begin{align*}
   \prox_{\varphi} (x) := \prox_{\varphi}^I (x). 
\end{align*}
Some basic properties of the proximity operator are summarized in the following result.

\begin{prop}\label{prop:proximity}
	Let $\varphi:\R^n\to\R$ be convex and $H\in\mathbb{S}_{++}^n$. Then 
	the following statements hold:
	\begin{enumerate}[label=(\alph*)]
		\item The proximity operator is firmly nonexpansive with respect to the norm induced by $H$ \cite[Lemma 3.1.1]{milzarek2016numerical}, i.e. for any $x,y\in\R^n$ there holds
		$$
			\big\| \prox_{\varphi}^H(x)-\prox_{\varphi}^H(y)\big\|_H^2\leq \big\langle \prox_{\varphi}^H(x)-\prox_{\varphi}^H(y), x-y\big\rangle_H.
		$$
		%\item The proximity operator $(x,H)\mapsto \prox_{\varphi}^H(x)$ is Lipschitz continuous on every compact subset of $\R^n\times\mathbb{S}_{++}^n$ \cite[Corollary 3.1.4]{milzarek2016numerical}.
		\item $p=\prox_\varphi^H(x)$ if and only if $p\in x-H^{-1}\partial \varphi(p)$ \cite[Proposition 16.44]{bauschke2017convex}.
	\end{enumerate}
\end{prop}

\noindent
Let $x,d\in\R^n$. Then, the directional derivative of $\psi$ in $x$ and direction $d$ is the one-sided limit
$$
	\psi'(x;d):=\lim_{t\downarrow0}\frac{\psi(x+td)-\psi(x)}{t}.
$$
We call $x^*\in\R^n$ a \emph{stationary point} of $\psi$ or a \emph{stationary point} of problem \eqref{eq:problem} if $0\in\nabla f(x^*)+\partial\varphi(x^*)$. Thus, we obtain the following characterizations:
\begin{align}
	x^*\text{ stationary point of }\psi \quad&\quad\Longleftrightarrow -\nabla f(x^*)\in\partial\varphi(x^*)\notag\\
	&\quad\Longleftrightarrow \psi'(x^*;d)\geq 0\text{ for all }d\in\R^n\label{eq:stationarypoint}\\
	&\quad\Longleftrightarrow x^*=\prox_\varphi^H(x^*-H^{-1}\nabla f(x^*)),\notag
\end{align}
where the second line follows from \cite[Proposition 17.14]{bauschke2017convex} and the final one is a consequence of Proposition \ref{prop:proximity}(b), which is independent of the particular matrix $H\in\mathbb{S}_{++}^n$. Given $x\in\R^n$ and $H\in\mathbb{S}_{++}^n$, it follows that the norm
of the corresponding residual
\begin{align*}
	r_H(x) :=&\underset{d}{\arg\min}\Big\{\nabla f(x)^Td+\frac 12 d^THd+\varphi(x+d)\Big\}=\prox_\varphi^H\big(x-H^{-1}\nabla f(x)\big)-x
\end{align*}
can be used to measure the stationarity of $ x $. For the special case $ H = I $, we again
simplify the notation and write $ r(x):=r_I(x)$. 
The relation between $ \| r_H(x) \| $ and
$ \| r_{\tilde{H}} (x) \| $ for two different matrices $ H, \tilde{H} $ is stated in the next result.

\begin{lem}\label{lem:tseng}
	Let $x\in\R^n$ and $H,\tilde{H}\in\mathbb{S}_{++}^n$. Then
	$$
	\|r_{\tilde{H}}(x)\|\leq \bigg( 1+\frac{\lambda_{\max}(\tilde{H})}{\lambda_{\min}(H)}\bigg)\cdot\frac{\lambda_{\max}(H)}{\lambda_{\min}(\tilde{H})}\cdot \|r_H(x)\|.
	$$
\end{lem}

\begin{proof}
	By \cite[Lemma 3]{tseng2009coordinate}, we get
	$$
	\|r_{\tilde{H}}(x)\|\leq\frac{1+\lambda_{\max}(Q)+\sqrt{1-2\lambda_{\min}(Q)+
	\lambda_{\max}(Q)^2}}{2}\ \frac{\lambda_{\max}(H)}{\lambda_{\min}(\tilde{H})}\cdot \|r_H(x)\|,
	$$
	where $Q:=H^{-1/2}\tilde{H}H^{-1/2}$ is also positive definite. The claim follows from the inequalities
	$$
		1-2\lambda_{\min}(Q)+\lambda_{\max}(Q)^2\leq 1+\lambda_{\max}(Q)^2\leq (1+\lambda_{\max}(Q))^2
	$$
	and $\lambda_{\max}(Q)\leq \lambda_{\max}(\tilde{H})/\lambda_{\min}(H)$. The latter estimate
	follows from 
	\begin{align*}
	   \lambda_{\max}& (Q)  = \max_{x \neq 0} \frac{x^T H^{-1/2} \tilde{H} H^{-1/2} x}{x^T x} 
	   = \max_{z \neq 0} \frac{z^T \tilde{H} z}{z^T H z}
	   = \max_{z \neq 0} \bigg( \frac{z^T \tilde{H} z}{z^T z} \frac{z^T z}{z^T H z} \bigg) \\
	   & \leq \bigg( \max_{z \neq 0} \frac{z^T \tilde{H} z}{z^T z} \bigg)
	   \bigg( \max_{z \neq 0} \frac{1}{\frac{z^T H z}{z^T z}} \bigg) 
	   = \lambda_{\max} (\tilde{H}) \frac{1}{\min_{z \neq 0} \frac{z^T H z}{z^T z}} 
	   = \lambda_{\max} (\tilde{H}) \cdot \frac{1}{\lambda_{\min} (H)},
	\end{align*}
	and this completes the proof.
\end{proof}

\section{The Regularized Proximal Quasi-Newton Method}\label{sec:method}

This section contains a detailed derivation and discussion of our regularized proximal quasi-Newton method.
Given an iterate $x^k\in\R^n$, consider the subproblem
\begin{equation}\label{eq:ClassicalSubproblem}
   \min_d q_k(d) \quad \text{with} \quad q_k(d):=f(x^k)+\nabla f(x^k)^T d+\tfrac 12 d^TB_kd+\varphi(x^k+d),
\end{equation}
where the first part is a quadratic approximation to the smooth function $f$, with $B_k$ being
a (possibly bad) approximation of the (possibly not existing) Hessian $\nabla^2 f(x^k)$. The main idea of proximal quasi-Newton methods is then to compute $ d^k $ as a solution of the subproblem 
\eqref{eq:ClassicalSubproblem}, and to 
set $ x^{k+1} := x^k + d^k $ provided that $ d^k $ is accepted by a suitable line search
or trust-region strategy in order to obtain global convergence results. Here, the 
globalization is done by a regularization parameter, no line search is required
(which might result in many function evaluations), and no trust-region radius is
needed (in particular, no trust-region-type subproblem has to be solved).
Instead, however, additional evaluations of the proximity operator may be required, which can be quite expensive. Nevertheless, numerical tests show that this additional effort leads to significantly fewer iterations and thus lower overall costs, and, furthermore, trust-region methods are more appropriate, especially for non-convex global optimization problems.

The regularized proximal quasi-Newton method therefore considers the regularized approximation
\begin{equation}\label{eq:hatq_definition}
	\hat{q}_k(d):=q_k(d)+\tfrac 12\mu_k \|d\|^2 = f(x^k)+\nabla f(x^k)^T d+\tfrac 12 d^T(B_k+\mu_k I)d+\varphi(x^k+d)
\end{equation}
with some parameter $\mu_k>0$. To control the success of a candidate $d^k$, which is a solution of 
the regularized subproblem $\min_d \hat{q}_k(d)$, we define the \emph{predicted reduction} of $\psi$ as
$$
	\pred_k:=\psi(x^k)-q_k(d^k)=-\big(\nabla f(x^k)^Td^k+\varphi(x^k+d^k)-\varphi(x^k)\big)-\tfrac 12 (d^k)^TB_kd^k
$$
and the \emph{actual reduction} of $\psi$ as $\ared_k:=\psi(x^k)-\psi(x^k+d^k)$.
The ratio $\rho_k:=\ared_k/\pred_k$ between these quantities is, similar to trust-region methods \cite{conn2000trust}, used to control the update of the regularization parameter and the iterate. Since $B_k$ does not need to be positive definite, we have to take into account that a minimizer of $\hat{q}_k$ may not exist 
or the corresponding value $\pred_k$ is not (sufficiently) positive. These situations are handled as unsuccessful steps. Altogether, this motivates the following algorithm.

\begin{alg} [Regularized Proximal Quasi-Newton Method] \label{alg:rpqnm} 
\leavevmode \vspace{-1.3\baselineskip}
\begin{itemize}
   \item[(S.0)] Choose $x^0\in\mathbb{R}^n$, parameters $\mu_0>0$, $p_{\min}\in(0,\tfrac 12)$, $c_1\in (0,\tfrac 12)$, $c_2\in(c_1,1)$, $\sigma_1\in(0,1), \sigma_2>1$, and set $ k := 0 $. %, and $\kappa > 1$.
		\item[(S.1)] If $ x^k $ satisfies a suitable termination criterion: STOP.
		\item[(S.2)] Choose $B_k\in\mathbb{R}^{n\times n}$, and find a solution $d^k$ of the problem
  		\begin{equation}\label{eq:algorithm-subproblem}
			\min_d \hat{q}_k(d) = f(x^k)+\nabla f(x^k)^T d+\tfrac 12 d^T(B_k+\mu_k I)d+\varphi(x^k+d).
		\end{equation}
		If this problem has no solution, or if 
		\begin{equation}\label{eq:pred_condition}
		   \pred_k\leq p_{\min}\|d^k\|\cdot \|r(x^k)\|,
		\end{equation}
		set $x^{k+1}:=x^k$, $\mu_{k+1}:=\sigma_2 \mu_k$, and go to (S.4). Otherwise go to (S.3).
		\item[(S.3)] Set $\rho_k:=\ared_k/\pred_k$ and perform the following updates:
		$$
			x^{k+1}:=\begin{cases} x^k&\text{if }\rho_k\leq c_1 , \\ 
			x^k+d^k&\text{otherwise,} \end{cases}\qquad 
			\mu_{k+1}:=\begin{cases} \sigma_2\mu_k&\text{if }\rho_k\leq c_1, \\\mu_k&\text{if }
		        c_1<\rho_k\leq c_2, \\ \sigma_1\mu_k&\text{otherwise}.\end{cases}
		$$
		\item[(S.4)] Update $k\leftarrow k+1$, and go to (S.1).
	\end{itemize}
\end{alg}

\noindent
In the following, we call an iteration $k$
\begin{itemize}
	\itemsep0pt
	\item \emph{unsuccessful}, if (S.3) is skipped or $\rho_k\leq c_1$,
	\item \emph{successful,} if $c_1<\rho_k\leq c_2$,
	\item \emph{highly successful,} if $\rho_k>c_2$.
\end{itemize}
Note that, in an unsuccessful iteration, both (S.2) and (S.3) keep the current iterate $ x^k $ and choose a 
larger regularization parameter. In all other iterations, we update $ x^{k+1} $ and either keep
the regularization parameter $ \mu_k $ (in all successful iterations) or reduce this
parameter (in all highly successful iterations). We also stress that a test like 
\eqref{eq:pred_condition} is not required by trust-region methods since, there, the 
corresponding predicted reduction is automatically positive, whereas this cannnot be
guaranteed in our setting. Whenever we reach (S.3), however, the value of $\pred_k$ is
(sufficiently) positive, which, in turn, implies that the overall method is well-defined.
\bigskip

\noindent
We briefly discuss the differences between Algorithm \ref{alg:rpqnm} and some affiliated methods.
The methods in \cite{ghanbari2018proximal,scheinberg2016practical} are based on a similar regularization than ours, where the regularization parameter is only increased if a suitable criterion is not satisfied for the solution of the subproblems. In contrast to our method, they do not consider the possibility to reduce the regularization parameter if an iterate is highly successful. Convergence is shown under the assumption of strong convexity of $f$. Furthermore, they combine the method with an inexactness criterion on the subproblem and  use a FISTA-type acceleration. In this case, a main assumption on $f$ is convexity.

The method by Karimi and Vavasis \cite{karimi2017imro} is a basic proximal Newton method for solving $\ell_1$-regularized least squares problems. No regularization is included and their analysis focusses on $H_k$ being a rank-1 modification of a multiple of the identity.

The inexact algorithms by Lee and Wright \cite{lee2019inexact} use two different types of regularization: $H_k=B_k+\mu_kI$ or $H_k=\mu_kB_k$ with a positive regularization parameter $\mu_k$, which is initially set to 1 in each step and increased until a sufficient decrease condition is satisfied. In contrast to our method, it is not possible to choose $\mu_k$ small when the iterate is close to a solution. Convergence is shown for $ \nabla f $
being Lipschitz continuous (but $f$ is not necessarily convex). Moreover, some improved convergence results are provided for strongly convex functions.

Yue et al.\ \cite{yue2019family} develop another inexact regularized proximal Newton method. A main difference to our approach is that, instead of an approximation $B_k$, the exact Hessian of $f$ is used and the regularization parameter $\mu_k$ is chosen due to the optimality of the current iterate, and not based on the quality of the current update. Furthermore, the subproblems are solved inexactly, and an Armijo-type line search is performed. The convergence proof needs convexity of $f$ and uses an error bound.

In contrast to these methods, we do not provide a theory for inexact solutions of the subproblems in (S.2). It turns out that this is not necessary since these problems can be solved very efficiently and with high accuracy in our numerical examples.
\bigskip

\noindent
In view of \eqref{eq:stationarypoint}, we know that $x^k$ is a stationary point of $\psi$ if and only if $r(x^k)=0$. Combining this property with the (uniform) continuity of $r(\cdot)$ yields an appropriate termination criterion for Algorithm \ref{alg:rpqnm}. For the method to be well-defined, we need a similar property for the value $d^k$ (note that, by definition, we have $ d^k = r_{B_k + \mu_k I} (x^k) $, if the matrix $B_k+\mu_k I$ is positive definite).

\begin{lem}\label{lem:dequals0}
	If $d^k=0$ in Algorithm \ref{alg:rpqnm}, then $x^k$ is a stationary point of $\psi$.
	The converse is true if $B_k+\mu_kI$ is positive definite.
\end{lem}

\begin{proof}
	Assume that $d^k=0$. From the definition of $d^k$ and Fermat's rule, we get
	$$
		0\in \nabla f(x^k)+(B_k+\mu_kI)d^k+\partial\varphi(x^k+d^k).
	$$
	Plugging in $d^k=0$ yields $0\in\nabla f(x^k)+\partial\varphi(x^k)$, which is the desired result. Conversely, let $B_k+\mu_kI$ be positive definite and $x^k$ a stationary point of $\psi$. Then $-\nabla f(x^k)\in\partial \varphi(x^k)$, which yields 
	$\varphi(x^k+d)\geq\varphi(x^k)-\nabla f(x^k)^Td$ for every $d\in\R^n$. Thus,
	\begin{align*}
	\hat{q}_k(0)&=f(x^k)+\varphi(x^k)\leq f(x^k)+\nabla f(x^k)^Td+\varphi(x^k+d)\\
	&\leq f(x^k)+\nabla f(x^k)^Td+\frac 12 d^T(B_k+\mu_kI)d+\varphi(x^k+d)=\hat{q}_k(d)
	\end{align*}
	for all $ d \in \R^n $.
	Hence, $d^k=0$ due to the uniqueness of the global minimum for $ B_k + \mu_k I $ being
	positive definite.
\end{proof}
\bigskip

\noindent
It is not difficult to see that the converse statement in Lemma \ref{lem:dequals0} may not
hold if $ B_k + \mu_k I $ is only positive semidefinite or indefinite. Hence, the termination criterion in (S.1) of Algorithm \ref{alg:rpqnm} should rely on $r(x^k)$ instead of $d^k$ as positive definiteness of $B_k+\mu_kI$ is not required.

\section{Global Convergence Theory}\label{sec:globalconvergence}

In this section, we investigate the global convergence properties of Algorithm \ref{alg:rpqnm}. Similar to convergence results for trust-region methods this means that $\liminf_{k\to\infty}\|r(x^k)\|=0$ or $\lim_{k\to \infty}\|r(x^k)\|=0$, depending on the assumptions. Using \eqref{eq:stationarypoint}, this implies that every accumulation point is a stationary point of $\psi$.
To prove this, we assume that Algorithm~\ref{alg:rpqnm} generates an infinite sequence $ \{ x^k \} $.
Though, formally, we did not specify the termination criterion in (S.1), any suitable
stopping criterion will include a test whether the current point $ x^k $ is already a stationary
point of the given optimization problem. Now, to simplify some of the subsequent phrases, we
therefore assume throughout this section that none of the iterations $ x^k $ is already a stationary point. Then, by Lemma \ref{lem:dequals0}, we have $ d^k \neq 0 $ for all $ k $.

The subsequent global convergence analysis of Algorithm \ref{alg:rpqnm} does not require the
matrices $ B_k $ to be good approximations of the corresponding (possibly not existing)
Hessians $ \nabla^2 f(x^k) $.
We only need that the sequence $ \{ B_k \} $ is bounded. Before presenting the two main
global convergence theorems, we establish some technical results.

\begin{lem}\label{lem:limit}
    Let $ \{ B_k \} $ be a bounded sequence of symmetric matrices. Assume that $ \mu_k \to \infty $
    and $\{x^k\}\subset\R^n$ converges to a nonstationary point $\overline{x}$ of $\psi$. Then
	$$
	\underset{k\to\infty}{\lim\sup} \frac{\|r(x^k)\|}{\|r_{B_k+\mu_kI}(x^k)\|\cdot \mu_k}\leq 1.
	$$
\end{lem}

\begin{proof}
	The assumptions imply that $B_k+\mu_kI$ is positive definite for all sufficiently large $k$. Furthermore, $\|r_{B_k+\mu_kI}(x^k)\|\neq 0$ for sufficiently large $k\geq 0$ since $\overline{x}$ is not a stationary point of $\psi$ and $r$ is continuous. Thus, we can apply Lemma \ref{lem:tseng} with $H=B_k+\mu_kI$ and $\tilde{H}=I$ to get
	$$
		\frac{\|r(x^k)\|}{\|r_{B_k+\mu_kI}(x^k)\|}\leq \Big(1+\frac{1}{\lambda_{\min}(B_k)+\mu_k}\Big)\cdot\big(\lambda_{\max}(B_k)+\mu_k\big).
	$$
	Dividing this estimate by $\mu_k$, using the boundedness of the sequence $\{B_k\}$, and taking $k\to\infty$, it follows that the expression on the right-hand side tends to 1, which yields the claim.
\end{proof}

\noindent 
Recall that if $B_k+\mu_k I$ is positive definite, step $d^k$ can be written as $ d^k = r_{B_k + \mu_k I} (x^k) $. In the next result, we show that this sequence is a vanishing sequence under the assumptions that the sequence $\{\mu_k\}$ tends to $+\infty$ and $\{x^k\}$ is bounded.

\begin{prop}\label{prop:dtozero}
	Let $\{B_k\}$ be a bounded sequence of symmetric matrices. Assume that $\mu_k\to \infty$ and the sequence $\{x^k\}\subset\R^n$ generated by Algorithm \ref{alg:rpqnm} is bounded. Let $d^k:=r_{B_k+\mu_kI}(x^k)$. Then $d^k\to 0$.
\end{prop}

\begin{proof}
	Note that the boundedness of the sequence $\{B_k\}$ and $\mu_k\to\infty$ imply that $d^k$ is well defined for sufficiently large $k$. Moreover, the definition of successful steps implies that the sequence $\{\psi(x^k)\}$ is a monotonically decreasing. Hence, for all $k\in\mathbb{N}$ sufficiently large, we have
		\begin{align*}
		\psi(x^0) & \geq \psi(x^k) 
		= \hat{q}_k(0)
		\geq \hat{q}_k(d^k) \\
		&= f(x^k)+\nabla f(x^k)^Td^k+\frac 12 (d^k)^T(B_k+\mu_k I)d^k+\varphi(x^k+d^k)\\
		&\geq f(x^k)+\nabla f(x^k)^Td^k+\frac 12 (d^k)^T(B_k+\mu_k I)d^k+\varphi(x^k)+( u^k)^Td^k
		\end{align*}
	for some $u^k\in\partial \varphi(x^k)$. Since, by assumption, the sequences $\{x^k\}$ and $\{B_k\}$ are bounded and, therefore, the sequences $\{f(x^k)\}$, $\{\varphi(x^k)\}$, $\{\nabla f(x^k)\}$, and $\{u^k\}$ are bounded by the continuity of $f$, $\varphi$ and $\nabla f$ and Propositon \ref{prop:subdifferential}~(a), the limiting behaviour of the right-hand side is dominated by the quadratic term $\tfrac 12 (d^k)^T(B_k+\mu_k I)d^k$. Thus, this term is bounded from above, and the assumption $\mu_k\to \infty$ immediately implies $d^k\to 0$.
	
%	Let $M>0$ and $k_0\in\mathbb{N}$ such that $B_k+\mu_k I\succeq  MI$ holds for all $k\geq k_0$. Furthermore, let $X$ be a compact set such that $x^k\in X$ for all $k\geq k_0$ and $\bar u^k\in\partial \varphi(x^k)$. 	
%	Then for any $k\geq k_0$
%	\begin{align*}
%	\psi(x^0) & \geq \psi(x^k) \\
%	&= \hat{q}_k(0)\geq \hat{q}_k(d^k) = f(x^k)+\nabla f(x^k)^Td^k+\frac 12 (d^k)^T(B_k+\mu_k I)d^k+\varphi(x^k+d^k)\\
%	&\geq f(x^k)+\nabla f(x^k)^Td^k+\frac M2\|d^k\|^2+\varphi(x^k)+(\bar u^k)^Td^k\\
%	&\geq \min_{x\in X} \psi(x) - \|d^k\| \max_{x \in X, u\in\partial\varphi(x)}\|\nabla f(x)+u\| +\frac M2\|d^k\|^2\\
%	&=\alpha_1-\alpha_2\|d^k\|+\frac M2\|d^k\|^2=:\tilde{q}(d^k),
%	\end{align*}
%	where $\alpha_1:=\min_{x\in X} \psi(x)$ and $\alpha_2:=\max_{x \in X, u\in\partial\varphi(x)}\|\nabla f(x)+u\|$. Hence, $\tilde{q}$ is coercive. In particular, the level set $\operatorname{lev}_{\leq \psi(x^0)}\tilde{q}$ is bounded, i.e., the sequence $\{d^k\}$ is bounded.
%	
%	By definition of $d^k$ and Fermat's rule, there exists $u^k\in\partial \varphi(x^k+d^k)$ such that
%	$$
%	0=\nabla f(x^k)+u^k+B_kd^k+\mu_kd^k.
%	$$
%	As the sequences $\{\nabla f(x^k)\}$, $\{u^k\}$ (cf.\ Proposition~\ref{prop:subdifferential} (b)) and $\{B_kd^k\}$ are bounded, the sequence $\{\mu_kd^k\}$ must also be bounded. From $\mu_k\to\infty$, we get $d^k\to 0$.
\end{proof}

\noindent
The following result will be applied to the 
situation where we have only finitely many successful iterations, i.e., where $ x^k $
stays constant eventually, say $ x^k = x^{k_0} $ for all $ k\geq k_0 $ and some
sufficiently large index $ k_0 \in \mathbb{N} $. We formulate this result in a slightly more general context and
assume that we have a nonstationary limit point $ \overline x $. To avoid any ambiguity in the 
notation, we write $ \bar d^k := r_{B_k + \mu_k I} (\overline x) $, although, in the subsequent application, we will eventually have $ \bar d^k = d^k $ since $ \overline x $ corresponds 
to $ x^{k_0} $ ($ = x^k $ for all $ k\geq k_0 $). 

\begin{lem} \label{lem:descentdirection}
    Let $ \{ B_k \} $ be a bounded sequence of symmetric matrices. Assume that $ \mu_k \to \infty $
    and $\overline{x}$ is a nonstationary point of $\psi$.
	Define $\bar d^k:=r_{B_k+\mu_kI}(\overline{x})$, and let $s$ be an accumulation point of the sequence 
	$\{\bar d^k/\|\bar d^k\|\}$. Then $\psi'(\overline{x};s)<0$.
\end{lem}

\begin{proof}
	Using the previous result, we get $\bar d^k\to 0$. Furthermore, using Fermat's rule, we obtain
	\begin{equation}\label{eq:techproof1}
		0=\nabla f(\overline x)+(B_k+\mu_kI)\bar d^k+u^k
	\end{equation}
	for some $u^k\in\partial \varphi(\overline{x}+\bar d^k)$. The boundedness of the subdifferential (Proposition~\ref{prop:subdifferential} (b)) yields that the sequence $\{u^k\}$ is bounded. Thus, we can choose a subsequence $K\subset\mathbb{N}$ such that
	$$
		\frac{\bar d^k}{\|\bar d^k\|}\to_K s\qquad\text{and}\qquad u^k\to_K \overline{u}.
	$$
	The closedness of the subdifferential (Proposition \ref{prop:subdifferential} (c)) yields $\overline{u}\in\partial \varphi(\overline{x})$. By assumption,
	we therefore have $\nabla f(\overline{x})+\overline{u}\neq 0$.
	
	Furthermore, using the results of \cite[Proposition 2.4]{lee2014proximal}, see also equation (2.16) in that paper, we obtain
	%Furthermore, using \cite[Lemma 2.3]{kanzow2021globalized}, we obtain 
	$$
		\psi'(\overline{x},\bar d^k)\leq-(\bar d^k)^T(B_k+\mu_kI)\bar d^k
		\leq - \big( \lambda_{\min} (B_k) + \mu_k \big) \| \bar d^k \|^2.
	$$
	Since \eqref{eq:techproof1} implies 
	$\|\nabla f(\overline{x})+u^k\|=\|(B_k+\mu_kI)\bar d^k\|\leq (\|B_k\|+\mu_k)\|\bar d^k\|$, we get
	\begin{equation*}
		\psi' ( \overline{x},\bar d^k ) \leq -\big(\lambda_{\min}(B_k)+\mu_k\big)\|\bar d^k\|^2\leq -\|\nabla f(\overline{x})+u^k\|\cdot\frac{\lambda_{\min}(B_k)+\mu_k}{\|B_k\|+\mu_k}\cdot \|\bar d^k\|.
	\end{equation*}
	Thus, the sublinearity of $\psi'(\overline{x},\cdot)$ yields
	$$
	\psi'\big(\overline{x},\frac{\bar d^k}{\|\bar d^k\|}\big) \leq -\|\nabla f(\overline{x})+u^k\|\cdot \frac{\lambda_{\min}(B_k)+\mu_k}{\|B_k\|+\mu_k}.
	$$
	For $k\to_K\infty$, the right-hand side converges to $-\|\nabla f(\overline{x})+\overline{u}\|$. 
	Since $\varphi$ is real-valued, the directional derivative $\psi'(\overline{x},\cdot)$ is continuous, and we obtain
	\begin{equation*}
	\psi'(\overline{x},s)=\underset{K\ni k\to\infty}{\lim} \psi'\Big(\overline{x},\frac{\bar d^k}{\|\bar d^k\|}\Big)\leq -\|\nabla f(\overline{x})+u\|<0.
	\end{equation*}
	This completes the proof.
\end{proof}
\bigskip

\noindent 
We now apply the previous result to show that there always exist infinitely many successful
or highly successful iterations.

\begin{lem}\label{lem:inf_successful}
    Let $ \{ B_k \} $ be a bounded sequence of symmetric matrices.
	Then Algorithm~\ref{alg:rpqnm} performs infinitely many successful or highly successful steps.
\end{lem}

\begin{proof}
	We follow the proof of \cite{steck2019regularization} and assume, by contradiction, that there exists $k_0\in\mathbb{N}$ such that all steps $k\geq k_0$ are unsuccessful. This implies $x^k=x^{k_0}$ for all $k\geq k_0$ and, due to the implicit assumption that Algorithm~\ref{alg:rpqnm} generates an infinite sequence,
that $ \mu_k \to + \infty $. Since $\{B_k\}$ is a bounded sequence, the matrices $B_k+\mu_kI$ are therefore positive definite for all sufficiently large $k$. In view of Lemma~\ref{lem:dequals0} and $ d^k \neq 0 $ (otherwise we would have stopped after finitely many iterations), it follows that $x^{k_0}$ is a nonstationary point of $\psi$, i.e., $r(x^{k_0})\neq 0$. Moreover, the positive definiteness of $B_k+\mu_kI$ also guarantees that the search directions $d^k$ are well-defined. In view of Lemma~\ref{lem:limit}, we have
	$$
		\frac{\| r(x^k) \|}{\|d^k\|\mu_k} < \frac{1}{2p_{\min}}
	$$
	for sufficiently large $k$ (recall that $p_{\min} < \tfrac 12$ and $ d^k = r_{B_k + \mu_k I} 
	(x^k) $). Using $\hat{q}_k(d^k)\leq \hat{q}_k(0)$, we then obtain
	\begin{align}
	\pred_k &=\psi(x^k)-q_k(d^k)=\psi(x^k)-\hat{q}_k(d^k)+\frac{\mu_k}{2}\|d^k\|^2\notag\\
	&\geq \psi(x^k)-\hat{q}_k(0)+\frac{\mu_k}{2}\|d^k\|^2 =\frac{\mu_k}{2}\|d^k\|^2%\notag\\	
	%&=-\big(\nabla f(x^k)^Td^k+\varphi(x^k+d^k)-\varphi(x^k)\big)-\frac 12 (d^k)^TB_kd^k\\
	%&\geq \frac 12 (d^k)(B_k+\mu_k I)d^k-\frac 12 (d^k)^TB_kd^k=\frac 12 \mu_k\|d^k\|^2\\
	 > p_{\min}\|r(x^k)\|\cdot \|d^k\|. \label{eq:pred_holds}
	\end{align}
	Hence, for all sufficiently large $k$, Algorithm \ref{alg:rpqnm} performs (S.3). Since all iterations $k\geq k_0$ are unsuccessful, this means $\operatorname{ared}_k\leq c_1\pred_k$. It follows that
	$$
	\psi(x^{k_0}+d^k)-\psi(x^{k_0}) \geq c_1\big(\nabla f(x^{k_0})^Td^k+\varphi(x^{k_0}+d^k)-\varphi(x^{k_0})+\tfrac 12 (d^k)^TB_kd^k\big).
	$$
	Setting $t_k=\|d^k\|$ and dividing this estimate by $t_k$ yields
	\begin{align*}
	&\frac{\psi(x^{k_0}+t_k\tfrac{d^k}{\|d^k\|})-\psi(x^{k_0})}{t_k} \\
	&\qquad\geq c_1\bigg( \nabla f(x^{k_0})^T\frac{d^k}{\|d^k\|}+\frac{\varphi(x^{k_0}+t_k\tfrac{d^k}{\|d^k\|})-\varphi(x^{k_0})}{t_k}+\frac 12 \frac{(d^k)^T}{\|d^k\|}B_kd^k\bigg).
	\end{align*}
	Choosing a subsequence $K$ such that $d^k/\|d^k\|\to s$, and using the local Lipschitz continuity of $\psi$, the left-hand side converges to the directional derivative $\psi'(x^{k_0};s)$ when taking the limit in $K$. In the same way, the limit of the second term on the right-hand side converges to $\varphi'(x^{k_0};s)$. 
	Thus, using $d^k\to 0$, see Proposition~\ref{prop:dtozero}, and the boundedness of $\{B_k\}$, taking the limit on $K$ in the entire estimate gives $\psi'(x^{k_0};s)\geq c_1 \psi'(x^{k_0};s)$. Since $ c_1 \in (0,1) $, this yields $ \psi'(x^{k_0};s) \geq 0 $, a contradiction to Lemma~\ref{lem:descentdirection}. This shows that there are infinitely many successful or highly successful iterations.
\end{proof}

\noindent 
We next formulate two global convergence results. The corresponding statements are similar
to those known for trust-region methods in, e.g., unconstrained optimization.

\begin{thm}\label{thm:global1}
Let $ \{ B_k \} $ be a bounded sequence of symmetric matrices, and assume that $ \psi $ 
is bounded from below. Then any sequence $ \{ x^k \} $ generated by the regularized proximal Newton-type method (Algorithm~\ref{alg:rpqnm})
satisfies $\lim\inf_{k\to\infty}\| r(x^k) \|=0$.
\end{thm}

\begin{proof}
	Let $\mathcal{S}\subset\mathbb{N}$ be the (infinite) set of successful or highly successful iterations. Contrary to the claim, assume that $\lim\inf_{k\to\infty} \|r(x^k)\|>0$. 
	Then there exists $k_0\in\mathbb{N}$ and $\varepsilon>0$ such that $ \|r(x^k)\| \geq \varepsilon$
	for all $ k \geq k_0 $. By the definition of successful steps, we get
	$$
	\psi(x^k)-\psi(x^{k+1})\geq c_1 \pred_k\geq p_{\min} c_1\|d^k\|\cdot \|r(x^k)\| \geq p_{\min}c_1\varepsilon\|d^k\|
	$$
	for all $k\in\mathcal{S},k\geq k_0$. Since $\psi$ is bounded from below, summation yields
	$$
	\infty>\sum_{k=0}^\infty \big[\psi(x^k)-\psi(x^{k+1})\big]=\sum_{k\in\mathcal{S}} \big[\psi(x^k)-\psi(x^k+d^k)\big]\geq p_{\min}c_1\varepsilon \sum_{k\in\mathcal{S}}\|d^k\|.
	$$
	Taking into account that $x^k$ is not updated in unsuccessful steps, it follows that
	$$
	\infty > \sum_{k\in\mathcal{S}} \|d^k\| = \sum_{k\in\mathcal{S}}\|x^{k+1}-x^k\|=\sum_{k=0}^\infty \|x^{k+1}-x^k\|.
	$$
	Hence, $\{x^k\}$ is a Cauchy sequence and therefore convergent to some $\overline{x}\in\mathbb{R}^n$. Since $\|r(\overline{x})\|=\lim_{k\to \infty}\|r(x^k)\|\geq\varepsilon$, $\overline{x}$ is not a stationary point of $\psi$.
	
	By Lemma \ref{lem:inf_successful}, there are infinitely many successful or highly successful steps
	and, as shown above, we have $\|d^k\|\to_{\mathcal{S}}0$. Similar to \eqref{eq:techproof1} there holds
	$$
	0=\nabla f(x^k)+(B_k+\mu_kI)d^k+u^k
	$$
	for some $u^k\in\partial\varphi(x^k+d^k)$. Assuming that $\{\mu_k\}_{\mathcal{S}}$ is bounded, $(B_k+\mu_k I)d^k$ converges to $0$ for $ k \to_{\mathcal{S}} \infty $. Furthermore, Proposition \ref{prop:subdifferential} (b), (c) yields that $\{u^k\}_{\mathcal{S}} $ is bounded and we can choose a subsequence $K\subset\mathcal{S}$ such that $u^k\to_K\overline{u}$ with $\overline{u}\in\partial\varphi(\overline{x})$.
	Taking the limit $K\ni k\to\infty$ in the above equation 
	then yields $0=\nabla f(\overline{x})+\overline{u}\in\nabla f(\overline{x})+\partial\varphi(\overline{x})$, in contradiction to the nonstationarity of $\overline{x}$.
	
	Hence, without loss of generality, we have $\{\mu_k\}_{\mathcal{S}}\to\infty$. It follows that $\{\mu_k\}\to\infty$ since $\mu_k$ cannot decrease during unsuccessful iterations. This implies that Algorithm~\ref{alg:rpqnm} also performs infinitely many unsuccessful iterations. On the other hand, in the same way as \eqref{eq:pred_holds}, we get
	$$
		\pred_k\geq p_{\min}\|d^k\|\cdot \|r(x^k)\| \geq p_{\min}\varepsilon \|d^k\|
	$$
for sufficiently large $k$. For every such $k$, there exists $\xi^k$ on the straight line between $x^k$ and $x^k+d^k$ such that $f(x^k+d^k)-f(x^k)=\nabla f(\xi^k)^Td^k$. By the convergence of $\{x^k\}$ to $ \overline{x} $ and since $\{d^k\} \to 0 $ in view of Proposition~\ref{prop:dtozero}, the sequence $\{\xi^k\}$ also converges to $\overline{x}$. Thus, we obtain
	\begin{align*}
	\big|\rho_k-1\big| &=\Big|\frac{\ared_k}{\pred_k}-1\Big| = \Big|\frac{\psi(x^k)-\psi(x^k+d^k)}{\psi(x^k)-q_k(d^k)}-1\Big|\\
	&=\Big|\frac{\psi(x^k+d^k)-q_k(d^k)}{\psi(x^k)-q_k(d^k)}\Big|\\
	&\leq \frac{1}{p_{\min}\varepsilon}\ \frac{\big| f(x^k+d^k)-f(x^k)-\nabla f(x^k)^T d^k\big|+\tfrac 12\big|(d^k)^T B_kd^k\big|}{\|d^k\|}\\
	&\leq \frac{1}{p_{\min}\varepsilon}\ \frac{\big| \nabla f(\xi^k)^T d^k-\nabla f(x^k)^Td^k\big|}{\|d^k\|}+\frac{1}{2p_{\min}\varepsilon} \bigg| (d^k)^TB_k\frac{d^k}{\|d^k\|} \bigg| \longrightarrow 0
	\end{align*}
	for $k\to\infty$. Hence, $\{\rho_k\}\to 1$, i.e., eventually all steps are successful or
	highly successful, which yields a contradiction.
\end{proof}

\noindent
Similar to trust-region methods, the previous result can be used to prove a stronger statement for  functions with a uniformly continuous gradient. The proof generalizes the one 
of \cite[Theorem 3.5]{steck2019regularization}.

\begin{thm}\label{thm:global2}
Let $ \{ B_k \} $ be a bounded sequence of symmetric matrices, assume that $\psi$ is bounded from below and that $\nabla f$ is uniformly continuous on a set $X$ satisfying $\{x^k\}\subset X$, where
$ \{ x^k \} $ denotes a sequence generated by Algorithm \ref{alg:rpqnm}. Then 
$\lim_{k\to\infty}\| r(x^k) \|=0$ holds; in particular, every accumulation point of $\{x^k\}$ is a stationary point of $\psi$.
\end{thm}

\begin{proof}
	Assume, by contradiction, that there exists $\delta>0$ and $K\subset\mathbb{N}$ such that $\|r(x^k)\|\geq 2\delta$ for all $k\in K$. By Theorem \ref{thm:global1}, for each $k\in K$, there is an index $\ell(k)>k$ such that $\|r(x^l)\|\geq\delta$ for all $k\leq l<\ell(k)$ and $\|r(x^{\ell(k)})\|<\delta$.
	
	If, for $k\in K$, an iteration $k\leq l<\ell(k)$ is successful or highly successful, we get
	$$
	\psi(x^l)-\psi(x^{l+1})\geq c_1\pred_l\geq 
        c_1 p_{\min} \| r(x^l) \| \cdot \| d^l \| \geq c_1 p_{\min} \delta \|x^{l+1}-x^l\|.
	$$
	For unsuccessful iterations $l$, this estimate holds trivially. Thus,
	$$
	p_{\min}c_1 \delta \|x^{\ell(k)}-x^k\|\leq p_{\min}c_1 \delta \sum_{l=k}^{\ell(k)-1}\|x^{l+1}-x^l\|\leq \sum_{l=k}^{\ell(k)-1} \psi(x^l)-\psi(x^{l+1}) = \psi(x^k)-\psi(x^{\ell(k)})
	$$
	holds for all $k\in K$. By assumption, $\psi$ is bounded from below, and by construction, the sequence $ \{ \psi (x^k) \} $ is monotonically decreasing, hence convergent. This implies
$\big\{\psi(x^k)-\psi(x^{\ell(k)})\big\}\to_K 0$. Hence, we get $\big\{\|x^{\ell(k)}-x^k\|\big\}\to_K 0$. The uniform continuity of $\nabla f$ and of the proximity operator (Proposition \ref{prop:proximity} (a)) together with the fact that the composition of uniformly continuous functions is uniformly continuous, yields the uniform continuity of the residual function $r(\cdot)$. Thus, we get $\big\{\|r(x^{\ell(k)})-r(x^k)\|\big\}\to_K 0$. On the other hand, by the choice of $\ell(k)$, we have
	$$
	\big\|r(x^k)-r(x^{\ell(k)})\big\|\geq \big\|r(x^k)\big\|-\big\|r(x^{\ell(k)})\big\|\geq 2\delta-\delta\geq \delta,
	$$
	which yields the desired contradiction.
\end{proof}

\section{Convergence Using an Error Bound Condition}\label{sec:error-bound-conv}

The aim of this section is to provide further convergence results for the regularized proximal quasi-Newton method in Algorithm \ref{alg:rpqnm}. To this end, we start with some technical results and then assume that $\nabla f$ is Lipschitz continuous to show the boundedness of the sequence $\{\mu_k\}$. Together with an error bound condition, we then deduce the convergence of the entire sequence. 
We start with some technical results.

\begin{lem}\label{50-lem:tech-estimates}
	Assume that the sequence $\{H_k\}$ is uniformly bounded and positive definite, i.e.\ there exist constants $0<m\leq M$ such that $mI\preceq H_k\preceq MI$ holds for all $k\geq 0$. Then the following estimates hold:
	\begin{enumerate}
		\item[(a)] $\displaystyle \pred_k\geq \frac 12 (m+2\mu_k)\|d^k\|^2$,
		\item[(b)] $\displaystyle \frac{\|r(x^k)\|}{\|d^k\|}\leq \Big(1+\frac{1}{m+\mu_k}\Big)(M+\mu_k)\leq \frac{m+1}{m}(M+\mu_k)$,
		\item[(c)] $\displaystyle \frac{\|d^k\|}{\|r(x^k)\|}\leq \frac{1+M+\mu_k}{m+\mu_k}\leq\frac{1+M}{m}$.
	\end{enumerate}
\end{lem}

\begin{proof}
	(a) Using \cite[Proposition 2.4]{lee2014proximal}, we get
	\begin{align*}
	\pred_k &=-\big(\nabla f(x^k)^Td^k+\varphi(x^k+d^k)-\varphi(x^k)\big)-\frac 12 (d^k)^TH_kd^k\\
	&\geq (d^k)^T(H_k+\mu_kI)d^k-\frac 12(d^k)^TH_kd^k\\
	&\geq \frac 12(m+2\mu_k)\|d^k\|^2.
	\end{align*}
	(b) and (c) follow directly from Lemma \ref{lem:tseng} using $\lambda_{\max}(H_k+\mu_kI)\leq M+\mu_k$ and $\lambda_{\min}(H_k+\mu_kI)\geq m+\mu_k$. 
\end{proof}

\noindent
The next result is essential to prove the boundedness of the sequence of regularizers $\{\mu_k\}$.

\begin{lem}\label{50-lem:mubig-success}
	Assume that $\nabla f$ is Lipschitz continuous with Lipschitz constant $L>0$ and $H_k\succeq mI$ for some $m>0$. If, in some iterate $x^k$, we have $\mu_k\geq \overline{\mu}:=\max\{L-m,0\}$, there holds $\ared_k>c_1 \pred_k.$
\end{lem}

\begin{proof}
	Let $\mu_k\geq \overline{\mu}$. Then $H_k+\mu_kI\succeq LI$, and the Lipschitz continuity of $\nabla f$ yields
	$$
	f(x^k+d^k)-f(x^k)\leq \nabla f(x^k)^Td^k+\frac 12 L\|d^k\|^2\leq \nabla f(x^k)^Td^k+\frac 12 (d^k)^T(H_k+\mu_kI)d^k,
	$$
	which is equivalent to
	$$
	\psi(x^k+d^k)-\psi(x^k)\leq \nabla f(x^k)^Td^k+\varphi(x^k+d^k)-\varphi(x^k)+\frac 12(d^k)^T(H_k+\mu_k I)d^k.
	$$
	Hence, using the definitions of $\pred_k$ and $\ared_k$, we get $-\ared_k\leq -\pred_k+\mu_k/2\ \|d^k\|^2$. A combination with Lemma \ref{50-lem:tech-estimates} (a) yields
	$$
	\ared_k\geq \pred_k-\frac{\mu_k}{2}\|d^k\|^2\geq \pred_k\cdot\frac{\mu_k+m}{2\mu_k+m}> \frac{1}{2}\pred_k\geq c_1\pred_k,
	$$
	which had to be shown (note that we need $c_1\leq \tfrac 12 $ at this point).
\end{proof}

\noindent
For the boundedness of the sequence $\{\mu_k\}$, it remains to prove that \eqref{eq:pred_condition} holds for sufficiently large $\mu_k>0$, which is the aim of the next result.

\begin{prop}\label{50-prop:mu-bounded}
		Assume that $\nabla f$ is Lipschitz continuous with Lipschitz constant $L>0$ and $MI\succeq H_k\succeq mI$ for some $M\geq m>0$. Then, the sequence $\{\mu_k\}$ generated from Algorithm \ref{alg:rpqnm} is bounded.
\end{prop}

\begin{proof}
	Assume that the sequence $\{\mu_k\}$ is unbounded. This means, there is a subsequence $K\subset \mathbb{N}_0$ such that $\{\mu_k\}_K\to\infty$. Since $\mu_k$ cannot increase in successful or highly successful steps, this implies that there are infinitely many unsuccessful steps. Without loss of generality we assume that all steps $k\in K$ are unsuccessful. In view of Lemma \ref{50-lem:mubig-success} this is only possible if for sufficiently large $k\in K$ we have $\pred_k < p_{\min}\|d^k\|\cdot\|r(x^k)\|.$
	Using Lemma \ref{50-lem:tech-estimates} (a), this yields
	$$
	\frac{m+2\mu_k}{2}\|d^k\|< p_{\min}\|r(x^k)\| \qquad\Longleftrightarrow\qquad \frac{\|r(x^k)\|}{\mu_k \|d^k\|}> \frac{m+2\mu_k}{2 p_{\min}\mu_k}.
	$$
	We combine this estimate with Lemma \ref{50-lem:tech-estimates} (b) to get
	$$
	\left(1+\frac{1}{m+\mu_k}\right) \frac{M+\mu_k}{\mu_k}> \frac{m+2\mu_k}{2 p_{\min}\mu_k}
	$$
	for $k\in K$. Taking the limit in $K$, the left hand side of this estimate converges to 1, whereas the right hand side converges to $1/p_{\min}>1$, which yields a contradiction. Hence, the sequence $\{\mu_k\}$ is bounded.
\end{proof}

\noindent
For the convergence of the complete sequence, we need an additional assumption.
In many papers the main assumption to prove local convergence and state a convergence rate is strong convexity. Here, more generally, we assume that $\psi$ satisfies a local error bound condition, which is used by Tseng and Yun in \cite{tseng2009coordinate}.

\begin{ass}\label{50-ass:error-bound}
	Assume that $\psi$ is bounded from below and $\mathcal{X}^*\neq \emptyset$, where $\mathcal{X}^*$ is the set of stationary points of $\psi$.
	\begin{enumerate}
		\item[(a)] For any $\zeta \geq \min_x \psi(x)$, there exist scalars $\tau>0$ and $\varepsilon>0$ such that
		$$
		\dist(x,\mathcal{X}^*)\leq \tau \|r(x)\|\quad \text{whenever}\quad \psi(x)\leq \zeta,\, \|r(x)\|\leq \varepsilon.
		$$
		\item[(b)] There exists a scalar $\delta > 0$ such that
		$$
		\|x-y\|\geq \delta \quad \text{whenever}\quad x\in \mathcal{X}^*, y\in\mathcal{X}^*, \psi(x)\neq \psi(y).
		$$
	\end{enumerate}
\end{ass}

\noindent
Similar assumptions to (a) have been investigated by Luo and Tseng in \cite{luo1992error,luo1993error}. Note that if a function satisfies the  above error bound condition, then it also satisfies the Kurdyka-\L ojasiewicz property \cite{li2018calculus}. Error bounds of this type have been studied by many authors, see e.g.\ \cite{yue2019family,zhou2017unified}.

Some examples of problem classes of the form \eqref{eq:problem} that satisfy Assumption \ref{50-ass:error-bound} (a) are, cf.\ \cite{tseng2009coordinate,yue2019family} and the references therein:
\begin{itemize}
	\itemsep0pt
	\item The function $f$ is strongly convex, $\nabla f$ is Lipschitz continous and $\varphi$ is an arbitrary convex function.
	\item $f(x)=h(Ax)+c^Tx$, where $h:\R^m\to\R$ is a continuously differentiable and strongly convex function such that $\nabla h$ is Lipschitz continuous on every compact set, $A\in\R^{m\times n},c\in\R^n$, and $\varphi$ has a polyhedral epigraph.
	\item $f(x)=h(Ax)$, where $A\in\R^{m\times n}$ and $h$ is given as above, and $\varphi(x)=\sum_{i=1}^s \|x_{G_i}\|_2$, where the sets $G_i\subset \{1,\dots,n\}$ form a partition of $\{1,\dots,n\}$ .
\end{itemize}
Many more functions of type \eqref{eq:problem} fulfill Assumption \ref{50-ass:error-bound} (a) even if they are not covered by the above problem classes. 
For more information and properties of error bound conditions, we refer to \cite{yue2019family,zhou2017unified,tseng2009coordinate}. 

Assumption \ref{50-ass:error-bound} (b) guarantees that the sets of stationary points of $\psi$ with different function values are properly separated. This assumption holds, in particular, if $\psi$ is convex.

It is important to note that we do not assume the convergence of the sequence $\{x^k\}$. Instead, this is a consequence of the above assumptions, as the following result shows. 

\begin{thm}\label{thm:error-bound-conv}
	Let $\{x^k\}$ be a sequence generated by Algorithm~\ref{alg:rpqnm} such that $\nabla f$ is Lipschitz continuous, $MI\succeq H_k\succeq mI$ for some $M\geq m>0$, and let Assumption \ref{50-ass:error-bound} hold. Then the sequence $\{x^k\}$ converges to some $\overline{x}\in\R^n$ and $\sum_{k=0}^\infty \|x^{k+1}-x^k\|<\infty$.
\end{thm}

This result is a simplified version of Theorem 2 in \cite{tseng2009coordinate} and, therefore, we skip the proof here. However, we briefly discuss the essential adaptations: First, the estimate of Lemma \ref{50-lem:tech-estimates} (c) in combination with Theorem \ref{thm:global2} yields $d^k\to 0$. Moreover, the crucial preliminary of \cite[Theorem 2]{tseng2009coordinate} is the boundedness of the analogous sequence to $\{B_k+\mu_k I\}$, which in our analysis is the result of the assumption on $\{B_k\}$ and Proposition \ref{50-prop:mu-bounded}. The further details of the proof are left to the reader.

\noindent
We note that it is also possible to develop a local convergence theory for Algorithm~\ref{alg:rpqnm} with small adjustments and under appropriate assumptions. In this paper, we focus on limited memory quasi-Newton approximations and therefore focus on the efficient solution of the related subproblems, which is the topic of the next section.

\section{Application to Limited Memory Proximal Quasi-Newton Methods}\label{sec:limited}

This section describes the central part for an efficient implementation of Algorithm~\ref{alg:rpqnm}
using limited memory matrices for $ B_k $. Since the idea itself is central for our work, we
first present the basic steps in a slightly simplified framework in Section~\ref{Sub:Idea}, and 
then come to the details for the actual realization in Section~\ref{Sub:Realization}.

\subsection{Main Idea Based on Compact Representations}\label{Sub:Idea}

The most costly part of Algorithm~\ref{alg:rpqnm} is the computation of $d^k$ in (S.2), which requires the solution of the minimization problem
$$
	\min_d f(x^k)+\nabla f(x^k)^Td+\tfrac 12 d^T(B_k+\mu_k I)d+\varphi(x^k+d).
$$
In the following, we assume that the matrix $B_k+\mu_k I$ is positive definite to ensure that the problem is solvable. If this is not the case, the problem might be unsolvable (depending on the properties of $\varphi$). Nevertheless, the following explanation mainly considers quasi-Newton matrices $B_k$ which fulfil this requirement under mild assumptions. If these are not met, the update is skipped.

So, if $B_k+\mu_kI$ is positive definite, we use the proximity operator to reformulate the problem to
\begin{equation} \label{eq:prox-subproblem}
	d^k = r_{B_k+\mu_k I}(x^k) = \prox_{\varphi}^{B_k+\mu_k I}\big( x^k- (B_k+\mu_k I)^{-1}\nabla f(x^k)\big)-x^k.
\end{equation}
Hence, the main effort is the computation of the proximity operator with respect to the norm induced by $B_k+\mu_k I$, where we are especially interested in the case that $B_k$ is obtained using a limited memory quasi-Newton update. The crucial point for that purpose consists in a suitable combination of a recent result by Becker et al.\ \cite{becker2019quasi} with the compact representation of limited memory quasi-Newton matrices introduced by Byrd et al.\ \cite{byrd1994representations}. 
We first describe the idea of our approach, and then provide the corresponding details for the actual realization (implementation) of the resulting method.

The class of quasi-Newton methods generates a sequence $ \{ x^k \} $ using the recursion 
$ x^{k+1} := x^k - H_k^{-1} \nabla f(x^k) $ for some suitable approximation $ H_k $ of the (not
necessarily existing) Hessian $ \nabla^2 f(x^k) $ (in our setting, we have $ H_k = B_k + \mu_k I $).
The matrices $ H_k $ are usually updated using rank-one or rank-two modifications; two well-known
examples are the SR1 update (symmetric rank-one)
\begin{equation*}
   H_{k+1} := H_{k+1}^{SR1} := H_k + \frac{(y^k - H_k s^k) (y^k - H_k s^k)^T}{(y^k - H_k s^k)^T s^k}
\end{equation*}
and the BFGS update (Broyden-Fletcher-Goldfarb-Shanno)
\begin{equation*}
   H_{k+1} := H_{k+1}^{BFGS} := H_k + \frac{y^k (y^k)^T}{(s^k)^T y^k} - 
   \frac{H_k s^k (s^k)^T H_k}{(s^k)^T H_k s^k} ,
\end{equation*}
where
\begin{equation*}
   s^k := x^{k+1} - x^k, \quad y^k := \nabla f(x^{k+1}) - \nabla f(x^k) \quad \forall k \in
   \mathbb{N}.
\end{equation*}
These quasi-Newton methods are not applicable to large-scale problems since the matrices $ H_k $
are dense. This problem can be avoided based on the following observation: The matrix
$ H_{k+1} $ can, in principle, be re-computed using the data $ H_0 $ together with the vectors
$ s^j $ and $ y^j $ for all $ j = 1, 2, \ldots, k $. Now, if we skip the first of these vectors
and use only the final $ m $ ones (for some small memory $ m \in \mathbb{N} $), we obtain a
limited memory quasi-Newton method, cf. \cite{nocedal1980updating}, which, due to a much smaller storage requirement, can be
applied to large-scale problems. These limited memory versions of standard quasi-Newton updates,
however, may not start with the same initial matrix $ H_0 $, instead they often use an initialization
$ H_{k,0} $ depending on the current iterate $ k $.

Now, consider the proximal subproblem
\begin{equation}\label{eq:Hksubproblem}
   \min_d f(x^k)+\nabla f(x^k)^Td+\tfrac 12 d^T H_k d+\varphi(x^k+d)
\end{equation}
for some suitable matrix $ H_k $. Using $ H_k := \lambda_k I$ $(\lambda_k>0)$, this subproblem is often easy to solve (sometimes
even analytically), whereas we obtain a much better approximation of the given
composite optimization problem if $H_k$ is chosen as a better approximation of the Hessian
$ \nabla^2 f(x^k) $, but then the subproblem itself is more difficult to solve. However, if
\begin{equation}\label{eq:Hksmallrank} 
   H_k = H_{k,0} + U_1 U_1^T - U_2 U_2^T 
\end{equation}
with suitable matrices $ U_i \in \mathbb{R}^{n \times r_i} $  (usually depending on $ k $, but to simplify the notation, we skip this index here) for some small $ r_i \in \mathbb{N} \ (i=1,2)$ and
a simple matrix $ H_{k,0} $ (typically a multiple of the identity matrix such that the corresponding
proximal subproblem is easy to solve), so that $ H_k $ is obtained from $ H_{k,0} $ by a small
rank-modification, then it is shown in Becker et al.\ \cite{becker2019quasi} that the solution of the difficult subproblem
\eqref{eq:Hksubproblem} can be computed from the solution of the (easy) proximal subproblem corresponding
to the matrix $ H_{k,0} $ using only some matrix-vector multiplications and solving a 
(strongly monotone, hence uniquely solvable) nonlinear system of equations of (small) dimension $ r_1+r_2 $.

Recalling the typical updates of quasi-Newton matrices, we immediately see that a single update of, e.g.,
the SR1- and the BFGS-method is precisely of the form required in \eqref{eq:Hksmallrank} with suitable matrices
$ U_1, U_2 $ of rank (at most) one. However, since the additive terms in these quasi-Newton updates
depend on $ H_k $ itself, these formulas cannot be used (directly) to apply the result from 
\cite{becker2019quasi}, which is based on the representation \eqref{eq:Hksmallrank}, to 
limited memories with $ m \geq 2 $. In fact, numerical
results presented in \cite{becker2019quasi} are based on taking a limited memory of $ m = 1 $ only. Their point is that for $m=1$ in the SR1-update, the occuring nonlinear system is of dimension 1 and can, hence be solved by bisection, and, if $\varphi$ is piecewise linear, even exact in log-linear time.

For many medium-sized problems, however, there are advantages to use a memory larger than 1.
This 
is the point where we can use the so-called compact representations of limited memory quasi-Newton matrices.

The Hessian approximations generated by most limited memory quasi-Newton methods can be written using a
compact representation of the form
\begin{equation}\label{eq:subs}
   H_k = H_{k,0} + A_k Q_k^{-1} A_k^T
\end{equation}
for some (usually diagonal) symmetric positive definite matrix $ H_{k,0} \in \mathbb{R}^{n \times n} $,
$ A_k \in \mathbb{R}^{n \times s} $, and a symmetric and nonsingular matrix $ Q_k \in 
\mathbb{R}^{s \times s} $, where, again, $ s \ll n $ is typically a very small number. Such a compact
representation can be used in order to rewrite $ H_k $ in a form required in \eqref{eq:Hksmallrank}.
To this end, we compute a spectral decomposition $ Q_k^{-1} = V_k \Lambda_k V_k^T $ of $ Q_k^{-1} $, i.e.,
$ V_k \in \mathbb{R}^{s \times s} $ is orthogonal and $ \Lambda_k \in \mathbb{R}^{s \times s} $ is
a diagonal matrix with diagonal entries $ \lambda_i^k $ (recall that $ s $ is small, hence the
computation of this spectral decomposition is not at all time-consuming). We then split the
diagonal matrix $ \Lambda_k $ into
\begin{equation*}
   \Lambda_k = \Lambda_k^+ - \Lambda_k^- ,
\end{equation*}
where $ \Lambda_k^+ $ and $ \Lambda_k^- $ are diagonal matrices consisting of the elements
$ \max \{ 0, \lambda_i^k \} $ and $ \max \{ 0, - \lambda_i^k \} $, respectively. Note that this
implies that these two diagonal matrices are positive semidefinite and, therefore, possess a 
matrix square root. Substituting this into 
\eqref{eq:subs} yields the representation \eqref{eq:Hksmallrank}
%\begin{equation*}
%   H_k = H_{k,0} + A_k V_k \Lambda_k^+ V_k^T A_k^T - A_k V_k \Lambda_k^- V_k^TA_k^T = 
%   H_{k,0} + U_1 U_1^T - U_2 U_2^T
% \end{equation*}
with the matrices (their dependence on $ k $ is neglected here)
\begin{equation*}
   U_1 := A_k V_k (\Lambda_k^+)^{1/2} \quad \text{and} \quad 
   U_2 := A_k V_k (\Lambda_k^-)^{1/2} .
\end{equation*}
Note that the two matrices $ U_1, U_2 $ actually simplify to some extent since some of their
columns are multiplied with zero entries of the corresponding diagonal matrices.
This completes the general description which allows an efficient implementation of our regularized
proximal limited memory quasi-Newton method.

\subsection{Realization of Proximal Subproblem Solutions}\label{Sub:Realization}

We now present the details of our realization of Algorithm~\ref{alg:rpqnm} where, we recall, we
have $ H_k = B_k + \mu_k I $ in the notation of the previous subsection, and where we use a 
limited memory update of $ B_k $ (not of $ H_k $ itself),
whereas the regularization term essentially only influences the initial matrix $ H_{k,0} $ (or $ B_{k,0} $
in our subsequent notation) since, in any case, this is typically just a multiple of the identity matrix.
Hence, assume we have a compact representation of the form
$$
	B_k = B_{k,0} + A_k Q_k^{-1}A_k^T,
$$
where $B_{k,0}\in\R^{n\times n}$ is a symmetric positive definite matrix, usually chosen as a multiple of the identity, $Q_k\in\R^{s\times s}$ is a symmetric and nonsingular matrix with $s\ll n$, and $A_k\in\R^{n\times s}$, cf.\ \cite{byrd1994representations}. The following example states explicitly the compact representations of the SR1- and the BFGS-updates, since these two will be exploited in our numerical experiments.

\begin{ex} \label{ex:limitedmemoryquasinewton}
As before, let $s^j = x^{j+1}-x^j$ and $y^j=\nabla f(x^{j+1})-\nabla f(x^j)$ for all $j$. Then, in iteration $ k $, we define the matrices 
$$
   S_k:=[s^{k-m}\dots s^{k-1}]\in\mathbb{R}^{n\times m} \quad\text{and}\quad Y_k:=[y^{k-m}\dots y^{k-1}]   
   \in\mathbb{R}^{n\times m}.
$$
Furthermore, let $D_k=D(S_k^TY_k)$ and $L_k=L(S_k^TY_k)$ denote the diagonal part and the strict lower triangle of the matrix $S_k^TY_k$.
	Then, the corresponding limited memory BFGS-update is given by the compact representation
	\begin{gather*}
		B_k := B_k^{BFGS} = B_{k,0} - \left[ \begin{matrix} B_{k,0}S_k & Y_k\end{matrix}\right] \left[ \begin{matrix} S_k^TB_{k,0}S_k&L_k\\L_k^T&-D_k\end{matrix}\right]^{-1} \left[\begin{matrix} S_k^TB_{k,0}\\Y_k^T\end{matrix}\right],\\
		\intertext{hence, }
		A_k = \left[ \begin{matrix} B_{k,0}S_k & Y_k\end{matrix}\right]\in\R^{n\times 2m}\quad\text{and}\quad Q_k = \left[ \begin{matrix} -S_k^TB_{k,0}S_k&-L_k\\-L_k^T&D_k\end{matrix}\right]\in\R^{2m\times 2m}.
	\end{gather*}
	Similarly, the limited memory SR1-update can be written as
	\begin{gather*}
	B_k := B_k^{SR1} = B_{k,0} + (Y_k-B_{k,0}S_k)(D_k+L_k+L_k^T-S_k^TB_{k,0}S_k)^{-1}(Y_k-B_{k,0}S_k)^T,\\
	\intertext{which yields}
	A_k = Y_k-B_{k,0}S_k\in\R^{n\times m}\quad\text{and}\quad Q_k = D_k+L_k+L_k^T-S_k^TB_{k,0}S_k\in\R^{m\times m},
	\end{gather*}
	see \cite[Theorems 2.3 and 5.1]{byrd1994representations}. \hfill$\Diamond$
\end{ex}

\noindent
To simplify the following discussion, we consider a fixed iteration $ k $ and therefore omit this index in the subsequent notation. 

Similar to Section~\ref{Sub:Idea}, with the matrix $ Q = Q_k $ available from the compact representation, we then compute a spectral decomposition $ Q^{-1} = V \Lambda V^T $ of $ Q^{-1} $ with $V\in\R^{s\times s}$ being orthogonal and $\Lambda \in\R^{s\times s}$ being a diagonal matrix. Let $\mathcal{I}_{1},\mathcal{I}_{2}\subset\{1,2,\dots,s\}$ be the sets of indices corresponding to the positive and negative entries of the diagonal of $\Lambda$, respectively. 

Define $\Lambda_{1}$ asthe submatrix of $\Lambda$ with the rows and columns in $\mathcal{I}_{1}$ and $\Lambda_2$ as the submatrix of $-\Lambda$ with the rows and columns in $\mathcal{I}_2$, and let $(AV)_1, (AV)_2$ be the submatrices of $A\cdot V$ with the column indices in $\mathcal{I}_{1}$ and $\mathcal{I}_2$, respectively. Then we can write
$$
	B = B_{0} + U_{1}U_{1}^T - U_{2}U_2^T
$$
with
\begin{equation}\label{eq:definition_U1U2}
	U_{1}:= (AV)_1 \Lambda_1^{1/2} \quad\text{and}\quad U_{2}:= (AV)_2 \Lambda_2^{1/2}. 
\end{equation}
Note that, by defining $\widehat{B_0}=B_0+\mu I$, we obtain a similar formula for the matrix $\widehat{B} =B+\mu I$.
At this point, we can use the following result from \cite[Corollary 3.6]{becker2019quasi} for the solution of \eqref{eq:prox-subproblem}.

\begin{thm} \label{thm:ochs}
	Let $\widehat{B}=\widehat{B}_0 + U_1U_1^T-U_2U_2^T \in\mathbb{S}_{++}^n$ with $\widehat{B}_0\in\mathbb{S}_{++}^n$ and $U_i\in\R^{n\times r_i}$ with rank $r_i$ ($i=1,2$). Set $\widehat{B}_1 = \widehat{B}_0+U_1U_1^T$. Then, the following holds:
	\begin{equation}
		\prox_{\varphi}^{\widehat{B}} (y)= \prox_{\varphi}^{\widehat{B}_0} (y + \widehat{B}_1^{-1} U_2\alpha_2^*-\widehat{B}_0^{-1}U_1\alpha_1^*),\label{eq:prox_ochs}
	\end{equation}
	where $\alpha_i^*\in\mathbb{R}^{r_i}$, $i=1,2$, are the unique zeros of the coupled system
$ \mathcal{L} (\alpha) = \mathcal{L} (\alpha_1, \alpha_2) = 0 $, where $ \mathcal{L} =
\big( \mathcal{L}_1, \mathcal{L}_2 \big) $ is defined by 
	\begin{align}
	\mathcal{L}_1(\alpha_1,\alpha_2) &= U_1^T(y+\widehat{B}_1^{-1}U_2\alpha_2-\prox_\varphi^{\widehat{B}_0}(y+\widehat{B}_1^{-1}U_2\alpha_2-\widehat{B}_0^{-1}U_1\alpha_1))+\alpha_1,\notag\\
	\mathcal{L}_2(\alpha_2,\alpha_2) &= U_2^T(y - \prox_\varphi^{\widehat{B}_0}(y+\widehat{B}_1^{-1}U_2\alpha_2-\widehat{B}_0^{-1}U_1\alpha_1))+\alpha_2.\label{eq:semismooth_system}
	\end{align}
\end{thm}

\noindent
In the following, we restrict the analysis to the case $B_{0}= \gamma I$ for some $\gamma>0$. Hence, in Theorem \ref{thm:ochs} we have $\widehat{B}_0 = \hat{\gamma}I$ with $\hat{\gamma}=\gamma+\mu$, which can be easily inverted and the proximity operator $\prox_\varphi^{\widehat{B}_0}$ can often be computed analytically. 
For the computation of $\widehat{B}_1^{-1}$ and $\widehat{B}^{-1}$, we use the Sherman-Morrison-Woodbury formula
to obtain
\begin{align*}
\widehat{B}_1^{-1} &= \hat{\gamma}^{-1} I - \hat{\gamma}^{-2}U_1(I+\hat{\gamma}^{-1}U_1^TU_1)^{-1}U_1
\qquad \text{and}\\
\widehat{B}^{-1} &= \widehat{B}_1^{-1} + \widehat{B}_1^{-1}U_2(I-U_2^T\widehat{B}_1^{-1}U_2)^{-1}U_2^T\widehat{B}_1^{-1}.
\end{align*}
Since the proximity operator is Lipschitz continuous, nonsmooth (semismooth) Newton methods are suitable candidates for the  numerical computation of the unique zero $\alpha^*=(\alpha_1^*,\alpha_2^*)$ of the nonlinear system of equations
$\mathcal{L}(\alpha) = 0 $ in Theorem~\ref{thm:ochs}. An iteration of the semismooth Newton method is given by 
\begin{equation}
	\alpha^{j+1} = \alpha^j - G_{j}^{-1}\mathcal{L}(\alpha^j),\label{eq:semismooth_update}
\end{equation}
where $G_j=G(\alpha^j)$ is a Newton derivative of $\mathcal{L}$ in $\alpha^j$, cf. \cite{qi1993nonsmooth}. For some details on Newton differentiable functions, we refer to \cite{griesse2008semismooth}. Provided that the Newton derivative of the proximity operator can be computed, a short calculation and the chain rule for generalized derivatives \cite[Theorem 3.5]{griesse2008semismooth} show the following result.

\begin{prop}
	Let $\prox_\varphi^{\widehat{B}_0}$ be Newton-differentiable with generalized derivative $P$. Then $\mathcal{L}$ is also Newton-differentiable, and the generalized derivative is given by
	\begin{equation}
		G(\alpha) = \left[\begin{matrix} U_1 & U_2\end{matrix}\right]^T P(z) \left[\begin{matrix} \widehat{B}_0^{-1}U_1 & -\widehat{B}_1^{-1}U_2\end{matrix}\right] + \left[\begin{matrix} I & U_1^T \widehat{B}_1^{-1}U_2\\0&I\end{matrix}\right],\label{eq:semismooth_derivative}
	\end{equation}
	where $z=y+\widehat{B}_1^{-1}U_2\alpha_2-\widehat{B}_0^{-1}U_1\alpha_1$.
\end{prop}

\noindent
In many applications the generalized derivative of the proximity operator can be computed analytically.

\begin{ex} \label{ex:proxgradients}
	(a) Let $\varphi(x):=\lambda\|x\|_1$ and $\widehat{B}_0=\hat{\gamma} I$ for some $\lambda,\hat{\gamma}>0$. Then the proximity operator is given (component-wise) by
	$$
	\big(\prox_{\varphi}^{\hat{\gamma} I}\big)_i (x) = 
	\begin{cases} x_i-\lambda\hat{\gamma},&\text{if } x_i\geq\lambda\hat{\gamma},\\
	0,&\text{if }  |x_i|<\lambda\hat{\gamma},\\
	x_i+\lambda\hat{\gamma},&\text{if } x_i\leq - \lambda\hat{\gamma}, \end{cases}
	$$
	cf. \cite[Example 3.2.8]{milzarek2016numerical}. Hence, the diagonal matrix $P(x)$ with diagonal 
	entries
	$$
		P_{ii}(x) = \begin{cases} 1, &\text{if }  |x_i|\geq \lambda\hat{\gamma},\\0,&\text{otherwise}\end{cases}
	$$
	is an element of the generalized Jacobian in the sense of Clarke, cf. \cite{clarke1975generalized}, and, therefore, a Newton derivative.
	
	(b) Let $\varphi(x):=\lambda\|x\|_2$. Then, an elementary calculation shows
	$$
	\prox_{\varphi}^{\hat{\gamma} I} (x) = x\cdot \max\Big\{1-\frac{\lambda\hat{\gamma}}{\|x\|_2},0\Big\} 
	%\begin{cases} x\big(1-\frac{\lambda\hat{\gamma}}{\|x\|_2}\big),&\text{if } \|x\|_2\geq\lambda\hat{\gamma},\\
	%0,& \text{otherwise},\end{cases}
	$$
	cf.\ \cite[Example 3.2.8]{milzarek2016numerical}. A short computation therefore shows that the following 	is a Newton derivative of this proximity operator:
	$$
	P(x) = \begin{cases} \big(1-\frac{\lambda\hat{\gamma}}{\|x\|_2}\big) I + \frac{\lambda\hat{\gamma}}{\|x\|_2^3}xx^T,&\text{if } \|x\|_2\geq\lambda\hat{\gamma},\\
	0,& \text{otherwise.}\end{cases} 
	$$
	The two examples given here will be used in our numerical section.\hfill$\Diamond$
\end{ex}

\noindent
We summarize the previous discussion and present our method for the computation of \eqref{eq:prox-subproblem} in the following algorithm.

\begin{samepage}
\begin{alg}[Solution of the subproblem \eqref{eq:prox-subproblem}] \label{alg:subproblem}
\leavevmode \vspace{-1.3\baselineskip}
\begin{itemize}
	\item[(S.0)] Given an iterate $x^k$, a compact representation $B_k=\gamma_k I + A_k Q_k^{-1} A_k^T$ of the corresponding Hessian approximation, and $\mu_k>0$.
		\item[(S.1)] Compute the spectral decomposition $Q_k^{-1} = V_k \Lambda_kV_k^T$, define 
		$$
			\mathcal{I}_1 :=\big\{ i\in\{1,\dots,s\}\mid \Lambda_k(i,i) >0\big\},\qquad \mathcal{I}_2 :=\big\{ i\in\{1,\dots,s\}\mid \Lambda_k(i,i) <0\big\},
		$$
		and determine $U_{1},U_2$ according to \eqref{eq:definition_U1U2}.
		
		\item[(S.2)] Choose $\alpha^0\in\R^{r_1+r_2}$ and compute the zero $\alpha^*$ of $\mathcal{L}=(\mathcal{L}_1,\mathcal{L}_2)$ defined in \eqref{eq:semismooth_system}, using a semismooth Newton method with the updates given in \eqref{eq:semismooth_update} and the generalized Jacobian given in \eqref{eq:semismooth_derivative}, until a suitable termination criterion holds.
		
		\item[(S.3)] Compute $d^k = \prox_{\varphi}^{B_k+\mu_k I}\big( x^k- (B_k+\mu_k I)^{-1}\nabla f(x^k)\big)-x^k$ using \eqref{eq:prox_ochs}.
	\end{itemize}
	
\end{alg}
\end{samepage}

\noindent
Of course, the most expensive part of Algorithm \ref{alg:subproblem} is the solution of the semismooth Newton equation in (S.2). While Becker et al. \cite{becker2019quasi} suggest a solution using an inexact semismooth Newton method in the general case, our experiments show that using the above described method performs just a few (in most cases 1-2) iterations to end up with an approximation of $\alpha^*$ satisfying $\|\mathcal{L}(\alpha^*)\|<10^{-10}$ independently of the size of the memory. This underlines the high efficiency of Algorithm \ref{alg:subproblem}, in particular using memories larger than one.

\section{Numerical Results}\label{sec:numeric}

In this section, we report numerical results for solving problem \eqref{eq:problem} using the Regularized Proximal Quasi-Newton Method (RPQN) from Algorithm \ref{alg:rpqnm} with limited memory quasi-Newton matrices. After comparing different limited memory methods for the computation of the occuring proximity operators, we compare this method with several methods applicable to solve problem \ref{eq:problem}.

The numerical results have been obtained in MATLAB R2020b using a machine running Open SuSE Leap 15.2 with
an Intel Core i5 processor 3.2 GHz and 16 GB RAM.

\subsection{Least Squares Problems with Group Sparse Regularizer} \label{sec:numexample1}

In our first example, we consider the least squares problem for $A\in\R^{m\times n}$ and $b\in\R^m$ with an $\ell_1$-$\ell_2$-sparsity regularizer, which is also called a group sparse regularizer in the literature. The problem is given by
$$
	\min_x \frac 12 \big\| Ax-b\|_2^2 + \lambda \|x\|_{2,1},
$$
where 
$$\|x\|_{2,1} := \sum_{j=1}^p \|x_{\mathcal{I}_j}\|_2.$$ 
Here, the index sets $\mathcal{I}_j$ $(j=1,\dots,p)$ form a partition of $\{1,\dots,n\}$. Since the groups $\mathcal{I}_j$ are pairwise disjoint, the proximity operator $\prox_{\lambda\|\cdot\|_{2,1}}$ and a Newton derivative thereof can be computed block-wise using the formulas in Example \ref{ex:proxgradients}. The use of the $\ell_1$-$\ell_2$-regularizer makes sense in many applications, where sparsity should be achieved with respect to some groups of variables. We refer to \cite{meier2008group} for more information about group (sparse) regularizers.

Note that the gradient $\nabla f(x)=A^T(Ax-b)$ of the function $f(x)=\tfrac 12 \|Ax-b\|_2^2$ is obviously Lipschitz continuous. Hence, the assumptions of Theorem \ref{thm:global2} are satisfied. Furthermore, by discussion in Section \ref{sec:error-bound-conv}, this problem setting also satisfies Assumption \ref{50-ass:error-bound} which, due to the convexity of the problem setting, implies the convergence of the complete sequence to a global minimizer.

\subsubsection{Problem Setting and Implementation}

We follow the generic example in \cite{becker2019quasi} and choose the entries in $A$ and $b$ from a uniform distribution in $[0,1]$ with $n=25k$ and $m=16k$ for various $k\in\mathbb{N}$. The parameter $\lambda$ is set to $1$. Furthermore, the index sets $\mathcal{I}_j$ are chosen randomly with 4 to 12 elements. The initial guess for the iterate is $x^0=0$.
In Algorithm~\ref{alg:rpqnm}, we choose the parameters $\mu_0 = 1$, $p_{\min}=c_1=10^{-4}$, $c_2=0.9$, $\sigma_1=0.5$ and $\sigma_2=4$. 

Furthermore, our tests showed that the semismooth Newton method for the computation of the proximity operators in Algorithm~\ref{alg:subproblem} converges very fast (mostly within 1 or 2 steps), so we stop if $\|\mathcal{L}(\alpha)\|<10^{-10}$ and use a maximal iteration number of 10. 
%Since the limited memory BFGS-updates are only well-defined if $(s^k)^Ty^k>0$, it is common to skip the update of the limited memory matrices if $(s^k)^Ty^k<\varepsilon \|s^k\|^2$ (a similar rule is used for the SR1-update). We choose $\varepsilon = 10^{-8}$ in our experiments and note that this case almost never occurs. 
Since the limited memory BFGS-updates are only well-defined if $(s^k)^Ty^k>0$, it is common to skip the update of the limited memory matrices if $(s^k)^Ty^k<\varepsilon \|s^k\|^2$.  For the SR1-update ill-conditioned steps are skipped easily in a similar way as described in \cite{byrd1994representations}: Instead of computing the spectral decomposition of $Q_k^{-1}$ in Algorithm \ref{alg:subproblem}, we compute the spectral decomposition $V_k\Lambda_kV_k^T$of $Q_k$ and define the index sets $\mathcal{I}_1$ and $\mathcal{I}_2$ to contain the indices such that $\Lambda_k(i,i)>\varepsilon$ and $\Lambda_k(i,i)<-\varepsilon$, respectively. With this strategy, rows and columns with ill-conditioned steps ($|\Lambda_k(i,i)|\leq \varepsilon$) are skipped.
We choose $\varepsilon = 10^{-8}$ in our experiments and note that updates are almost never skipped.
The initial estimate $\gamma_k$ for the computation of the limited memory quasi Newton matrices is set to 
$$
   \gamma_k = \frac{(y^k)^Ty^k}{(s^k)^Ty^k},
$$
following the approach of Liu and Nocedal \cite{liu1989limited}. 
There are several ways to update the matrix $B_k$ if a step was unsuccessful. In this case one could start again with memory 0. However, our experiments show better results if the update of $B_k$ is simply skipped.

To compare different methods, we initially run the algorithm once with a very high accuracy to determine a good approximation to the optimal function value $\psi^*$, and then terminate the methods if the current iterate $x^k$ satisfies
\begin{equation}\label{eq:numeric_error}
	\frac{\psi(x^k)-\psi(x^*)}{\max(1,|\psi(x^*)|)}\leq 10^{-6},
\end{equation}
where the term on the left hand side is referred to as \emph{objective value error}.
Besides analysing the regularized proximal quasi-Newton method (RPQN) itself, we compare it to the following methods:
\begin{itemize}
	\item QGPN (Globalized Proximal Quasi-Newton Method \cite{kanzow2021globalized})
	
	This method represents a class of several proximal quasi-Newton methods, which use an Armijo-type line search strategy to guarantee convergence. In contrast to other methods, e.g. \cite{becker2019quasi,lee2014proximal}, a further globalization using a proximal gradient method is applied, which has shown to improve the performance. Parameters are chosen as in \cite{kanzow2021globalized}. 
	%The subproblems are solved as described in Algorithm \ref{alg:subproblem} with a limited memory SR1-update of memory 5, which showed best performance in this problem class.
\end{itemize}
In addition to this second order proximal method, we use two well knows proximal first order methods to compare the results to. Although there are plenty of accelerated proximal first order methods, to the author's knowledge there is no clear favourite regarding the performance. Hence, we chose the following well-known ones. 

\begin{itemize}
	\item FISTA (Fast Iterative Shrinkage Thresholding Algorithm \cite{beck2009fast})
	
	FISTA is one of the most common accelerated first order methods for solving convex problems with composite functions. In every step a subproblem of the form \eqref{eq:prox-subproblem} is solved, where $B_k+\mu_kI$ is replaced by $L_kI$ and $L_k$ is an approximation to the Lipschitz constant of $\nabla f$. We start with the initial guess $L_0=1$ and increase with $\eta=2$, if the step is not successful.\\
	Although there are several adaptations of FISTA in the nonconvex setting, e.g.\ \cite{ochs2019adaptive}, we restrict the analysis to the convex version.
	
	\item SpaRSA (Sparse Reconstruction by Separable Approximation \cite{wright2009sparse})

	SpaRSA is another first order method for the considered problem class. The main difference to FISTA is the update of the factor $L_k$, which is done by a Barzilai-Borwein approach. Hence, the method is related to RPQN with a memory of $0$. Furthermore, the theory of SpaRSA also includes nonconvex functions.
\end{itemize}

\noindent
All that techniques are proximal-type methods, since these are highly efficient for solving optimization problems with composite functions. In the above setting, we also tested a method based on the forward backward envelope \cite{themelis2018forward}. Furthermore, the setting in the subsequent section allows using an interior point method, cf.\ \cite{kim2007interior}. However, these methods did not yield benefits in comparison to the above mentioned methods. Instead, we also provide comparisons with the following non-proximal method.

\begin{itemize}
	\item SNF (Semismooth Newton Method with Multidim. Filter Globalization \cite{milzarek2014,milzarek2016numerical})
	
	This method by Milzarek and Ulbrich is based on the semismooth Newton method to find a zero of $r(x)$, combined with a globalization using a filter strategy. There is a convex and nonconvex version of the filter conditions to decide whether the computed update is applied or a proximal gradient step is performed instead.
\end{itemize}

\subsubsection{Discussion of the Results}

\begin{figure}	
	\begin{minipage}[t]{0.46\textwidth}
		\vspace*{0pt}
		\includegraphics[width=\textwidth]{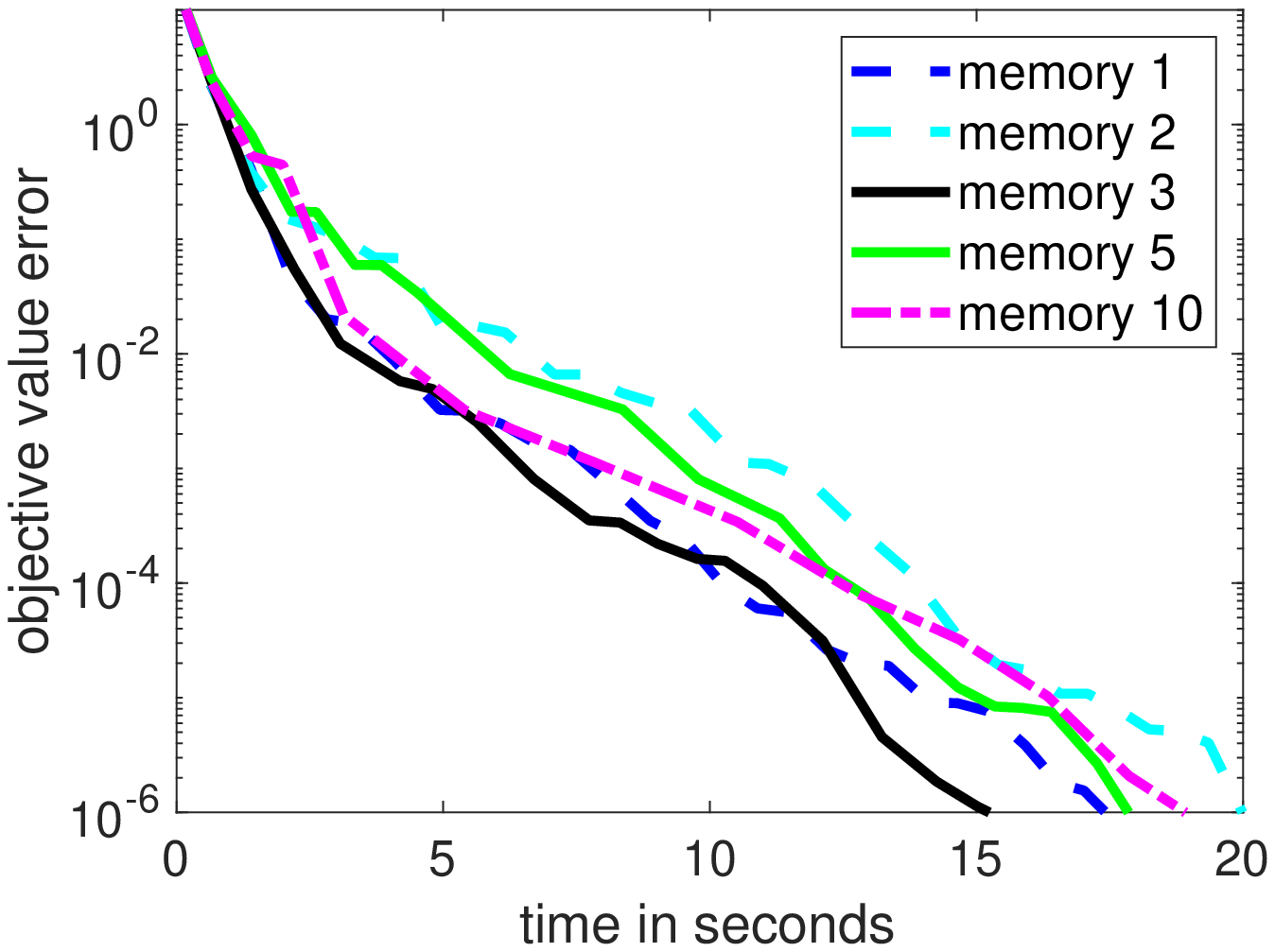}
		%\vspace*{-0.3cm}
	\end{minipage}
	\hfill
	\begin{minipage}[t]{0.46\textwidth}
	\caption[caption]{\small{\textsf{Convergence plot for RPQN with limited memory BFGS approach and different memories for the setting in Section \ref{sec:numexample1} for solving the least squares problem with group sparse regularizer. The run time is the average of 10 runs.}}\normalsize}
	\label{fig:group_lasso}
	\end{minipage}
\end{figure}	

We start comparing the size of the memory using the dimension $k=100$, i.e.\ $n=2500$ and $m=1600$, which should be chosen for the limited memory quasi-Newton method. Figure~\ref{fig:group_lasso} shows the relation between the elapsed run time and the current error as defined in \eqref{eq:numeric_error}, when RPQN is applied to the test problem with limited memory BFGS-updates. To avoid side effects and first-time computation costs, the time is averaged over 10 runs.
The choice of the memory size should be big enough to achieve good performance, but preferably small to save computation costs. Figure~\ref{fig:group_lasso} indicates that the impact of the memory size to the run time is relatively small, but the memory 3 showed the best performance. This is also validated by the data given in Table \ref{tab:num_vergleich}. In a similar test with limited memory SR1-updates, the best results were achieved with a memory of 5. RPQN with limited memory BFGS- and SR1-updates and the determined optimal memory sizes are denoted by RPQN (L-BFGS) and RPQN (L-SR1), respectively. 

\begin{table}[b]
	\begin{center}
		\small	
		\begin{tabular}{rcccccccc}
			%	\begin{tabular}{|r|c||c|c|c||c|c|c|c|}
			\hline
			method &iter& highly&succ.& unsucc.&sub-& function & proximity & matrix-vector\\ 
			(memory)& &s. iter&iter &iter &iter&eval &eval &products\\ \hline%\hline
			L-BFGS (1) &46&18&14&14&199&47&442&94\\ %\hline
			L-BFGS (2) & 36&18&5&13&149&36&333&73\\ %\hline
			L-BFGS (3) &49&27&6&16&208&50&461&100\\ %\hline
			L-BFGS (5)&55&32&3&20&265&53&577&106\\ %\hline
			L-BFGS (10)&34&20&2&12&121&33&276&66\\ \hline%\hline
			%LSR1 (5)&45&27&2&16&206&45&453&90\\  \hline 
		\end{tabular}
		
	\end{center}
	\vspace*{-0.5cm}
	
	\caption{\small \textsf{Values of the test example in Section~\ref{sec:numexample1} for the RPQN method with limited memory BFGS update and various memories.}
	}
	\label{tab:num_vergleich}
\end{table}

For a comparison to other state-of-the-art methods, we take $k\in\{1,3,10,30,100,300\}$ and run all algorithms on 10 random examples as described above. The average computation time in relation to the problem dimension is visualized in Figure \ref{60-fig:quadratic-group-sizes}. For the comparison we used RPQN and QGPN with limited memory BFGS-updates and a memory of 10. Note that QGPN did not converge within $10^4$ (outer) iterations for $n=7500$. One sees that the performance of the first-order methods is better for small problem sizes. This follows from the high computation costs for solving the subproblems, which does not yield a profit for small dimensions. On the other hand, starting with $n=750$, RPQN clearly outperforms the other methods, not only first-order, but also the tested second-order methods. This shows that the regularization in Algorithm \ref{alg:rpqnm} is superior although some iterations are unsuccessful and the computed solutions of the corresponding subproblems are discarded.

\begin{figure}[htb]
	\begin{center}
		\begin{minipage}{0.8\textwidth}
			\includegraphics[width=\textwidth]{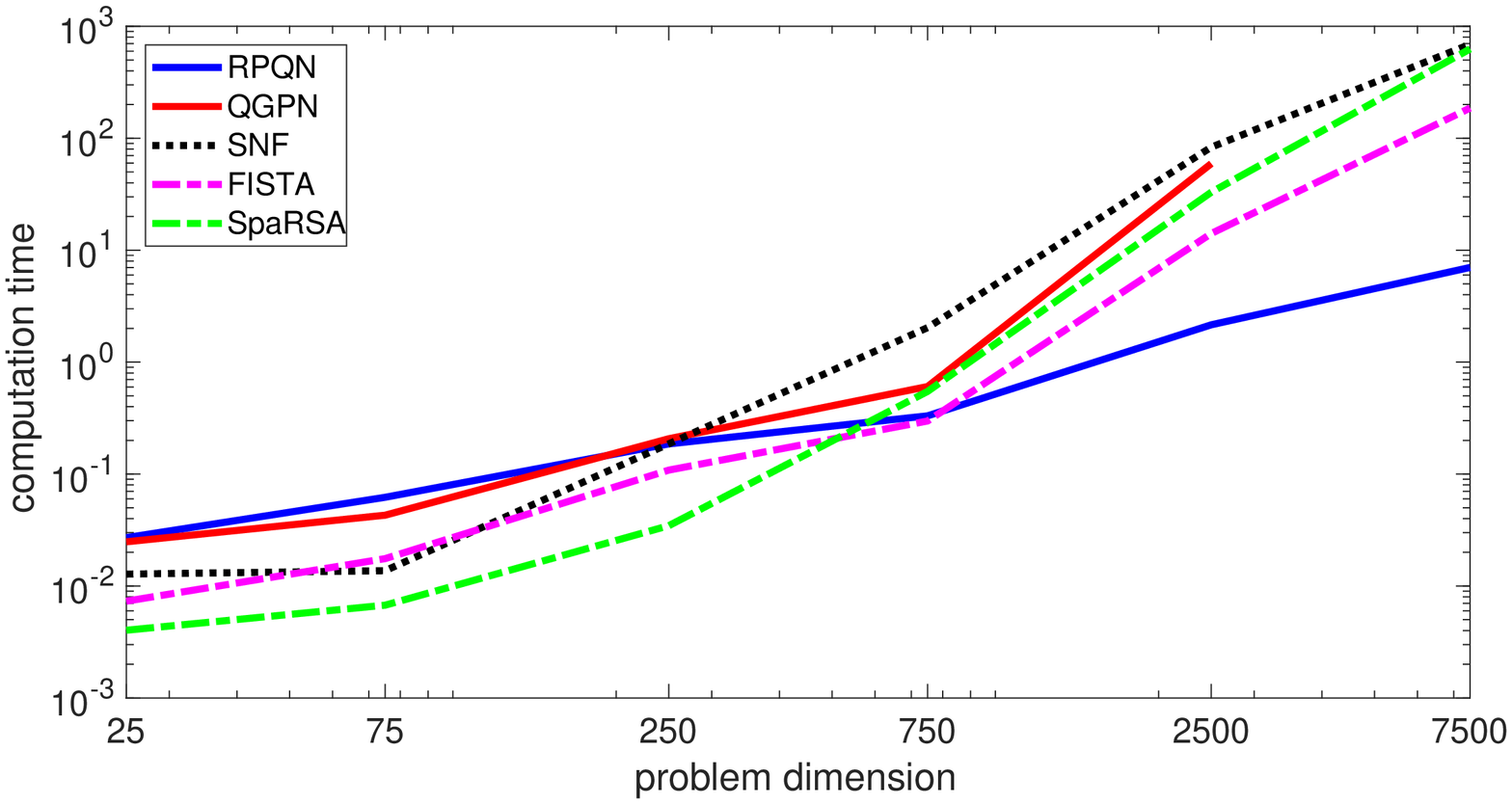}
			\vspace*{-0.8cm}
		\end{minipage}
	\end{center}
	\caption[caption]{\small \textsf{Comparison of the performance of several methods depending on the problem dimension as described in Section \ref{sec:numexample1}.}}
	\label{60-fig:quadratic-group-sizes}
\end{figure}

\subsection{\texorpdfstring{$\mathbf{\ell}_1$}{l1}-regularized Least Squares Problem (LASSO)} \label{sec:numexample2}

We demonstrate the performance of our method for the unconstrained LASSO (least absolute shrinkage and selection operator) problem
$$
	\min_x \frac 12\|Ax-b\|_2^2+\lambda\|x\|_1,
$$
with $A\in\R^{m\times n}, b\in\R^m$ and $\lambda > 0$. This formulation is used for many problems to handle sparsity when finding a solution of $Ax\approx b$, see e.g. \cite{figueiredo2007gradient,beck2009fast}. Again, we use a test setting from \cite{becker2019quasi} with $n=3000$ and $m=1500$, which is typical for compressed sensing. The entries of $A$ and $b$ are independently and identically distributed according to the standard normal distribution, the penalty parameter is chosen as $\lambda = 0.1$. We use the methods described in
Section~\ref{sec:numexample1} and almost all parameters are used as before, except that the memory for RPQN (L-BFGS) is set to $m=10$, and QGPN is applied with a limited memory BFGS-update and a memory of $5$, as these proved to be the best choices in our tests.

The results are illustrated in Figure~\ref{fig:num_vergleich23} (a). Again one sees that there is almost no difference between the optimal versions (concerning the size of the memory) of the limited memory BFGS- and SR1-updates of RPQN. Furthermore, these methods perform significantly better than the other tested methods. While QGPN can keep up until an objective value error of approximately $10^{-1}$, its performance gets very slow afterwards. The first order methods FISTA and SpaRSA have by far longer running times to achieve appropriate errors.

\begin{figure}[htb]	
	
	\begin{minipage}{0.46\textwidth}
		\includegraphics[width=\textwidth]{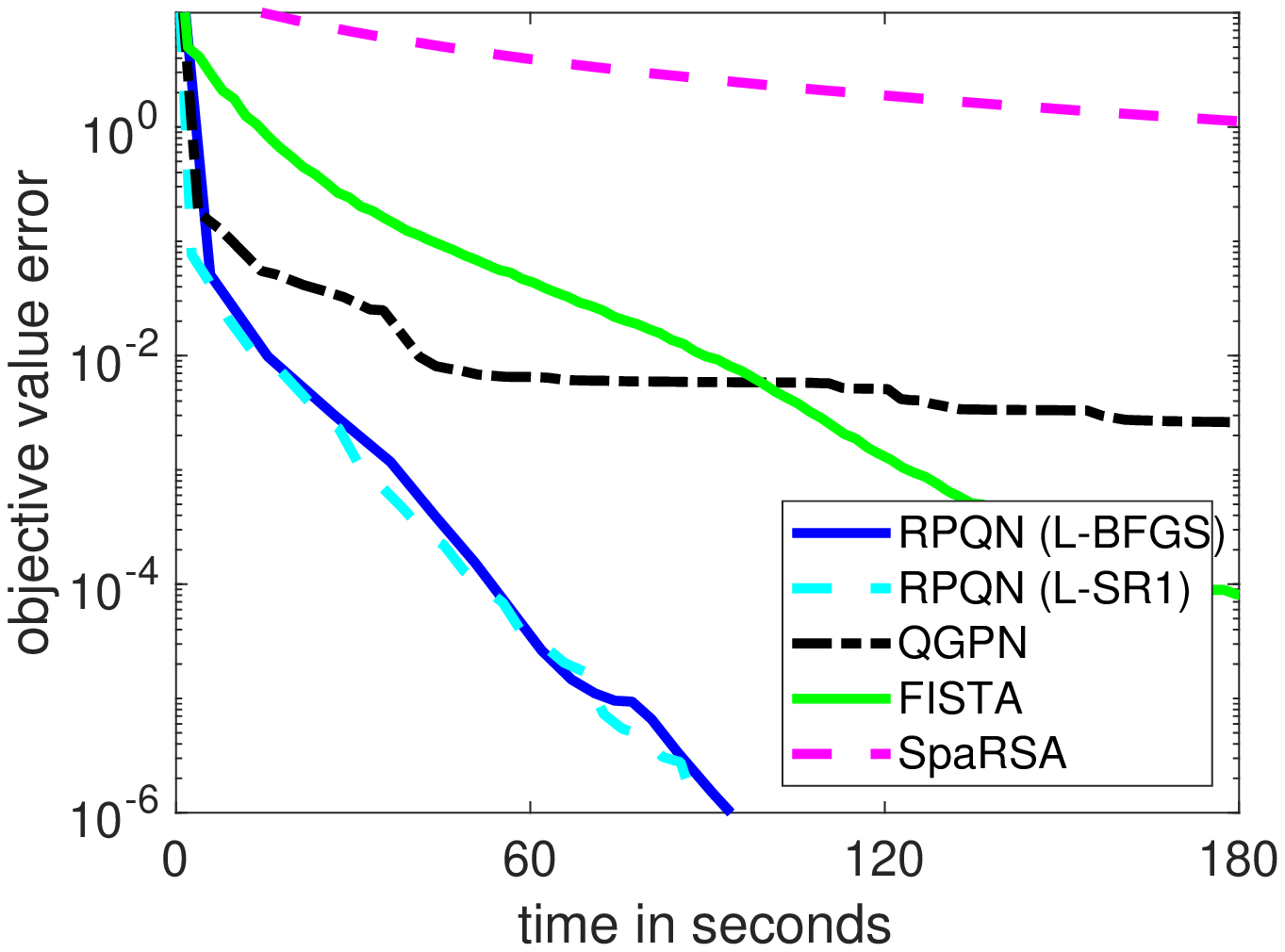}
		\vspace*{-0.8cm}
		
		\begin{center} \small \textsf{(a) Comparison of different methods for the example in Section~\ref{sec:numexample2}.}\end{center}
	\end{minipage}
	\hfill
	\begin{minipage}{0.46\textwidth}
		\includegraphics[width=\textwidth]{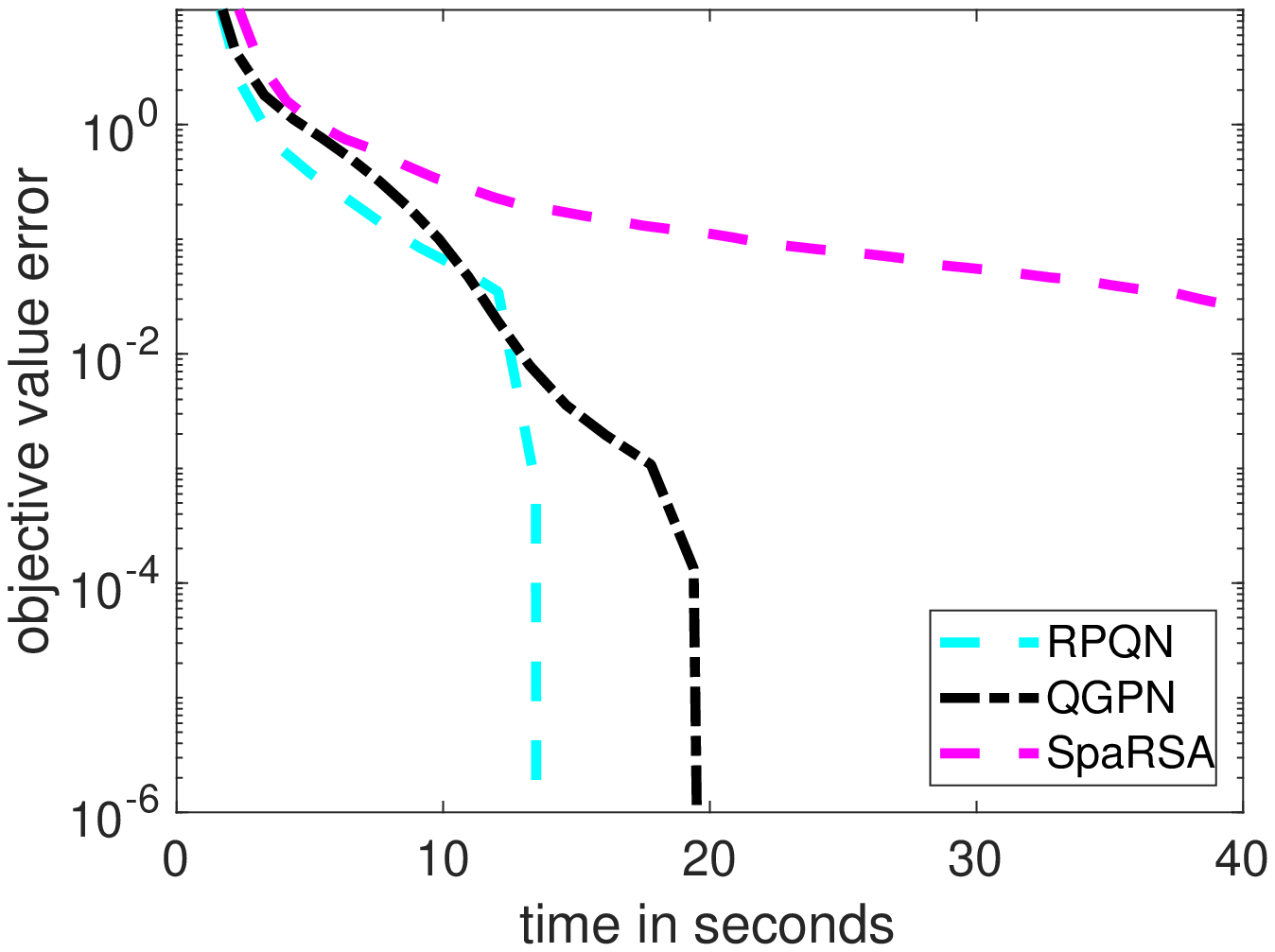}
		\vspace*{-0.8cm}
		
		\begin{center} \small \textsf{(b) Comparison of different methods for the example in Section~\ref{sec:numexample3}.}\end{center}
	\end{minipage}
	
	\caption[caption]{\small \textsf{Convergence plots for the $\ell_1$-regularized least squares problem (a) and the nonconvex image restoration (b). The run time is the average of 10 runs, the term "objective value error" refers again to the stopping criterion defined in \eqref{eq:numeric_error}.}}
	\label{fig:num_vergleich23}
\end{figure}

\subsection{Nonconvex Image Restoration} \label{sec:numexample3}

In this section, we consider a nonconvex image restoration problem. Given a noisy blurred image $b\in\R^n$ and a blur operator $A\in\R^{n\times n}$, the aim is to restore the original image $x\in\R^n$ such that $Ax\approx b$. If there are Gaussian errors on the image $b$, this problem can be solved efficiently using a quadratic loss similar to the previous sections. If the errors are distributed by Student's $t$-distribution, cf. \cite{aravkin2012robust2}, this approach usually does not perform well. For that purpose, the quadratic loss can be replaced by
$$
	f(x) := \sum_{i=1}^n \log\big((Ax-b)_i^2+1\big),
$$
cf. \cite{stella2017forward}. To guarantee antialiasing, we add the nonsmooth term $\varphi(x):=\lambda\|Bx\|_1$, where $B\in\R^{n\times n}$ is a two dimensional Haar wavelet transform and $\lambda>0$. Since $B$ is orthogonal, we can reformulate the problem $\min_x f(x)+\varphi(x)$ into
$$
	\min_y \sum_{i=1}^n \log\big((AB^Ty-b)_i^2+1\big)+\lambda\|y\|_1,
$$
where $y:=Bx$. The function $f$  is not convex, but  $\nabla f$ is Lipschitz continuous. Furthermore, we expect a solution to this problem to approximately fulfill $AB^Ty^*=b$, so $f$ is strongly convex in a neighbourhood of the solution if $A$ as full range. This means that our convergence theory applies here and we again get the convergence of the complete sequence of iterates to a stationary point.

\begin{table}[b]
	\begin{center}
		\small	
		\begin{tabular}{rccccccc}
			\hline
			method &iter& Newton- &succ.& sub-& function & proximity & matrix-vector\\ 
			& &iter&iter &iter &eval &eval &products\\ \hline%\hline
			RPQN&890&-&866&1790&891&4448&1790\\
			QGPN&1101&1098&-&1175&1113&2354&2215\\ 
			SNF&183&91&-&1189&784&408&3855\\
			SpaRSA&1089&-&-&1964&1965&1964&3930\\ \hline
		\end{tabular}
		
	\end{center}
	\vspace*{-0.5cm}
	
	\caption{\small \textsf{Numerical data for the image restoration example in Section~\ref{sec:numexample3}.
		}
	}
	\label{60-tab:image-data}
\end{table}

We follow the test setting in \cite{bot2016inertial_forward}, see also \cite{stella2017forward,kanzow2021globalized}, to restore a $256\times 256$ test image, hence $n=256^2 = 65 536$. The mapping $A$ is a Gaussian blur operator of size $9\times 9$ and with standard deviation $4$ and $B$ is the two dimensional discrete Haar wavelet of level $4$. Furthermore, we choose $\lambda = 10^{-4}$.
The noisy blurred image $b$ is created from the original cameraman image by applying $A$ and adding Student's $t$-noise with degree of freedom 1 and rescaled by $10^{-3}$, and we start with $y^0=b$.

For our analysis, we solve the image restoration with RPQN and QGPN with limited memory SR1-updates and a memory of 2 (which, again, behaved best in our tests), SNF and SpaRSA. Details on the methods are given in Section~\ref{sec:numexample1}. Note that we do not apply FISTA to this problem since this solver is designed for convex problems.

As before, using the same rules, we sometimes skip the limited memory updates. However, even though the problem is nonconvex and one can therefore expect that this case occurs more frequently, our experiments reveal that there is a maximum of one or two skipped updates per run of RPQN. 

\begin{figure}[tb]	
	\begin{center}
	\begin{minipage}[t]{0.27\textwidth}
		\includegraphics[width=\textwidth]{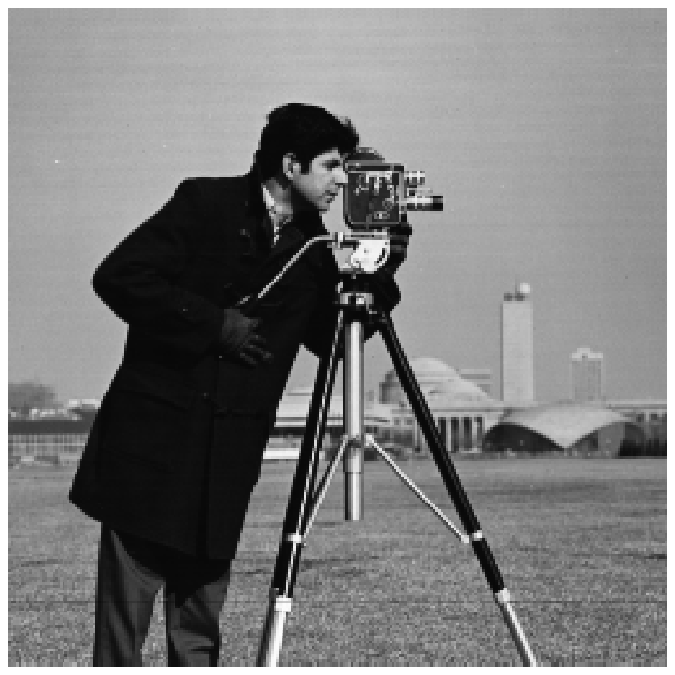}
		\vspace*{-0.8cm}
		
		\begin{center} \small \textsf{(a) Original Image}\end{center}
	\end{minipage}
	\begin{minipage}[t]{0.27\textwidth}
		\includegraphics[width=\textwidth]{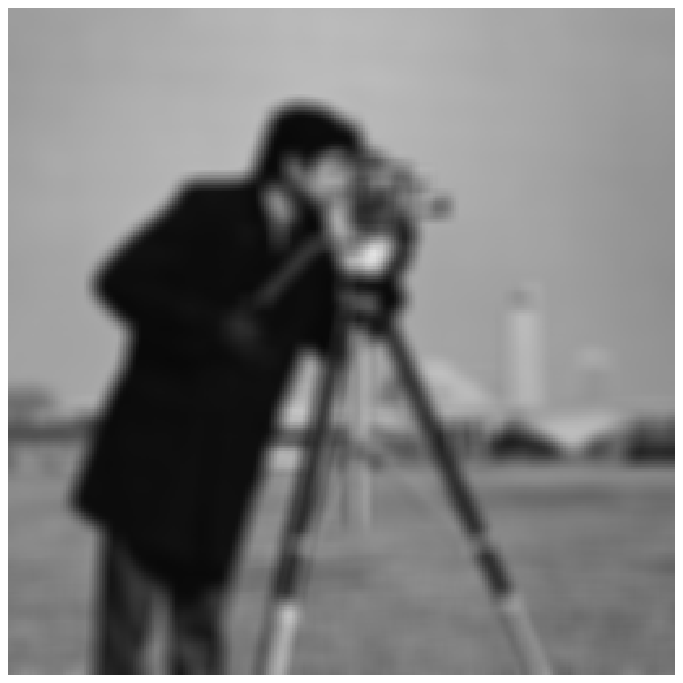}
		\vspace*{-0.8cm}
		
		\begin{center} \small \textsf{(b) Noisy Image}\end{center}
	\end{minipage}
	\begin{minipage}[t]{0.27\textwidth}
		\includegraphics[width=\textwidth]{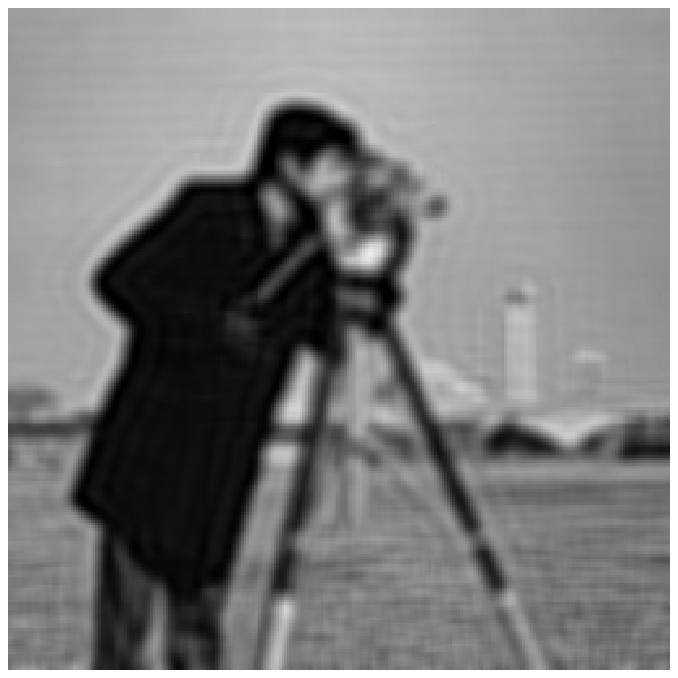}
		\vspace*{-0.8cm}
		
		\begin{center} \small \textsf{(c) SNF}\end{center}
	\end{minipage}
	
	\begin{minipage}[t]{0.27\textwidth}
		\includegraphics[width=\textwidth]{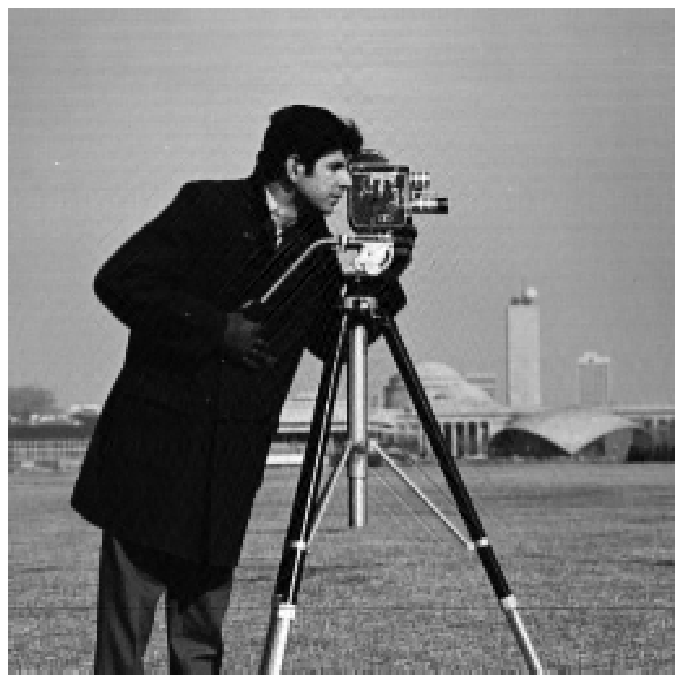}
		\vspace*{-0.8cm}
		
		\begin{center} \small \textsf{(d) RPQN}\end{center}
	\end{minipage}
	\begin{minipage}[t]{0.27\textwidth}
		\includegraphics[width=\textwidth]{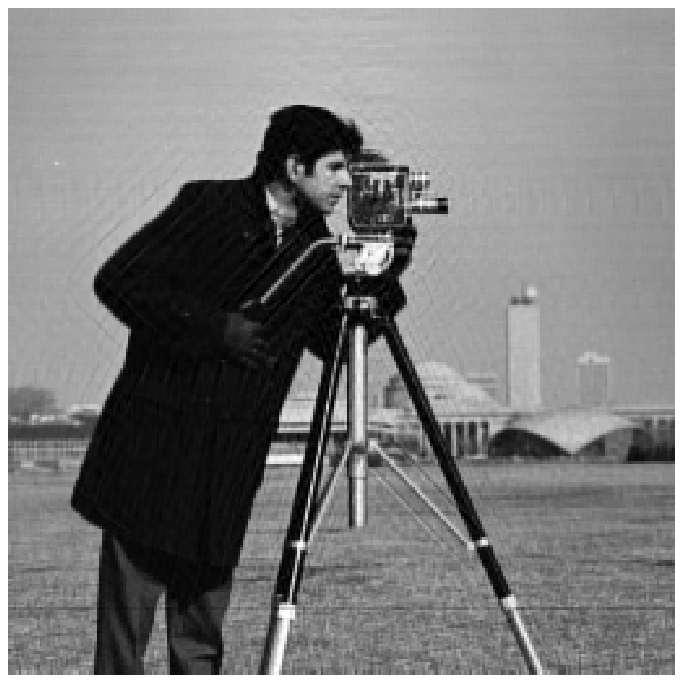}
		\vspace*{-0.8cm}
		
		\begin{center} \small \textsf{(e) QGPN}\end{center}
	\end{minipage}
	\begin{minipage}[t]{0.27\textwidth}
		\includegraphics[width=\textwidth]{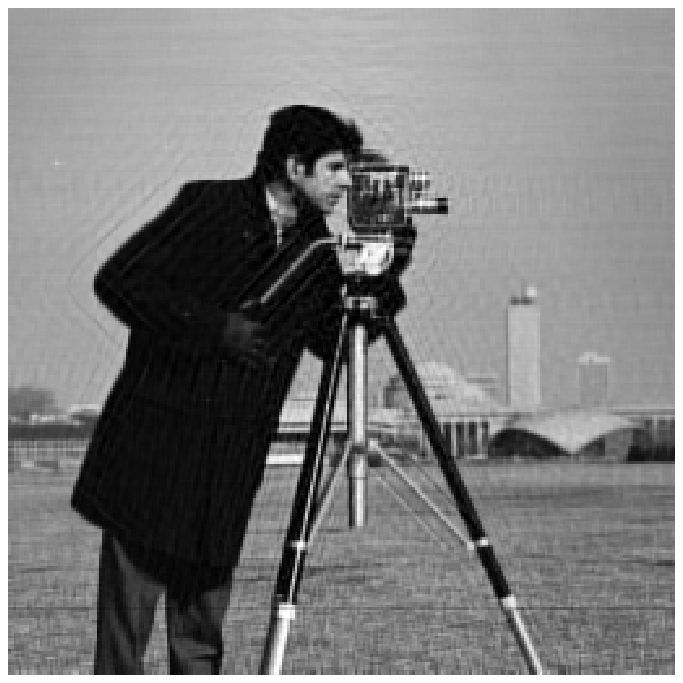}
		\vspace*{-0.8cm}
		
		\begin{center} \small \textsf{(f) SpaRSA}\end{center}
	\end{minipage}
	\caption[caption]{\small \textsf{Nonconvex image restoration from Section \ref{sec:numexample3}: Original and noisy image and recovered images using the stated algorithms and terminated after a computation time of 12 seconds.}}
	\label{fig:cameraman_results}
	\end{center}%\vspace*{-25pt}
\end{figure}

Here, we do not compute $\psi^*$ as the optimal value of the objective function, but as the function value of the original image (which are not the same in this case). For that reason, we terminate the methods if $\psi(x^k)\leq \psi(x^*)$ holds for an iterate $x^k$. The results, again averaged over 10 runs, are shown in Figure~\ref{fig:num_vergleich23} (b). For the first iterations, all methods show similar performance and there are only minor differences. At some point, however, RQPN and shortly after QGPN instantly satisfy the termination criterion, whereas SpaRSA performs several more iterations until this goal is reached. Note that the performance of SNF is not satisfactory in this example and not shown in Figure \ref{fig:num_vergleich23} (b). In the nonconvex setting, this might be due to the structure, where semismooth iterations reducing $\|r(x^k)\|$ but probably increasing $\psi(x^k)$ and proximal gradient iterations, which decrease $\psi(x^k)$ but probably increase $\|r(x^k)\|$ are expected to alternate.
We report some of the resulting data in Table \ref{60-tab:image-data}.

Looking at the performance in Figure \ref{fig:num_vergleich23} (b), we also display the resulting images of the tested methods after a computation time of 12 seconds (and not using the above termination criterion) in Figure \ref{fig:cameraman_results}. It can be observed that RPQN and QGPN restore the image relatively well, while the result of SpaRSA is also satisfactory, but SNF is clearly outperformed.

\section{Final Remarks}\label{sec:conclusion}

In this paper, we proposed a proximal quasi-Newton method with a regularization technique for a globalization, and presented the corresponding global convergence theory. After that we described a very efficient method for the computation of the occurring proximity operators using compact representations of limited memory quasi-Newton matrices. The numerical results show that the regularized method in combination with the efficient proximity operator computation accelerates the performance and outperforms both some standard first-order and some second-order methods.

Since our focus was on the limited memory quasi-Newton approach, we only presented a global convergence theory. A future approach is therefore to develop local convergence results  under appropriate assumptions including a convergence assumption on the matrices $B_k$.

Furthermore, a main issue is the assumption that the convex function $\varphi$ is real-valued,
and this fact is exploited in several steps of the current analysis.
In the authors' opinion, the deduced algorithm should perform well also for problems with extended-valued functions $\varphi$. Thus, a main task of future research is the investigation of the convergence theory for this class of functions.

Finally, the computation of the variable metric proximity operators can be investigated. Many authors \cite{yue2019family,ghanbari2018proximal,scheinberg2016practical,bonettini2017convergence} provide convergence results for inexact solutions of this problem in the setting of their proposed methods.  Although our experiments reach very high accuracies in solving the subproblems within a very few steps, an improvement of the presented method could be to connect it to some of these criteria.

% Literaturverzeichnis
	%\nocite{*}
	%\newpage
	%\selectlanguage{english}
	\bibliographystyle{siamplain}
	\small
	\bibliography{literatur_komplett_210801}

\end{document}